\documentclass{amsart}
\usepackage{amssymb}
\usepackage{amsfonts}
\usepackage{amsmath}
\usepackage{amsthm}
\usepackage{amssymb}
\usepackage[square,sort,comma,numbers]{natbib}
\usepackage{cases}
\usepackage{tcolorbox}
\usepackage{tikz}
\usepackage{tkz-graph}
\usepackage{indentfirst}
\usepackage{float}
\usepackage{enumitem}
\usepackage{csquotes}

\tikzstyle{vertex1}=[auto=left,circle,draw,fill=black!25,minimum size=10pt,inner sep=2pt]
\tikzstyle{Dot}=[auto=left,circle,fill=black!100,minimum size=2pt,inner sep=1.5pt]
\tikzstyle{dot}=[auto=left,circle,fill=black!100,minimum size=0.5pt,inner sep=0.5pt]
\tikzstyle{Dot2}=[auto=left,circle,fill=black!100,minimum size=1pt,inner sep=1pt]

\newtheorem{Definition}{Definition}
\newtheorem{Theorem}{Theorem}

\newtheorem{Corollary}{Corollary}
\newtheorem{proofpart}{Part}

\newcounter{cases}
\newcounter{subcases}
\newenvironment{mycases}
{%
	\setcounter{cases}{0}%
	\def\case
	{%
		\par\noindent
		\refstepcounter{cases}%
		\textbf{Case \thecases.}
	}%
}
{%
	\par
}
\newenvironment{subcases}
{%
	\setcounter{subcases}{0}%
}
{%
}
\renewcommand*\thecases{\arabic{cases}}

\newcommand*{\mybox}[1]{\framebox{#1}}

\newcounter{boxlblcounter}  
\makeatletter
\@addtoreset{proofpart}{Theorem}
\newcommand{\G}{\Gamma}
\newcommand{\V}{{\mathcal{V}}}
\newcommand{\E}{\mathcal{E}}
\newcommand{\W}{\mathcal{R}}
\newcommand{\U}{\mathcal{R}}
\newcommand{\C}{\mathcal{C}}
\renewcommand{\P}{\mathcal{P}}
\newcommand{\R}{\mathcal{R}}
\newcommand{\A}{a}
\newcommand{\XX}{x}
\newcommand{\YY}{y}
\newcommand{\RR}{\gamma}

\makeatletter
\newsavebox\myboxA
\newsavebox\myboxB
\newlength\mylenA

\newcommand*\xoverline[2][0.75]{%
	\sbox{\myboxA}{$\m@th#2$}%
	\setbox\myboxB\null% Phantom box
	\ht\myboxB=\ht\myboxA%
	\dp\myboxB=\dp\myboxA%
	\wd\myboxB=#1\wd\myboxA% Scale phantom
	\sbox\myboxB{$\m@th\overline{\copy\myboxB}$}%  Overlined phantom
	\setlength\mylenA{\the\wd\myboxA}%   calc width diff
	\addtolength\mylenA{-\the\wd\myboxB}%
	\ifdim\wd\myboxB<\wd\myboxA%
	\rlap{\hskip 0.5\mylenA\usebox\myboxB}{\usebox\myboxA}%
	\else
	\hskip -0.5\mylenA\rlap{\usebox\myboxA}{\hskip 0.5\mylenA\usebox\myboxB}%
	\fi}
\makeatother

\providecommand{\keywords}[1]{\textbf{\textit{Keywords---}} #1}

\begin{document}

\leftline{{\footnotesize }}

\centerline{}
\centerline{}
\centerline{}

\title[Metric dimensions of bicyclic graphs and potential
applications]{Metric dimensions of bicyclic graphs with potential applications in Supply Chain Logistics}
\author[M. Wang]{Muwen Wang}
\address{School of Business Administration Shandong Women's University, Jinan, China}
\email{wangmuwen@outlook.com}
\author[G. Haidar]{Ghulam Haidar}
\address{Department of Mathematics and Statistics, The University of Haripur, Pakistan}
\email{haidersehani012@gmail.com}
\author[F. Yousafzai]{Faisal Yousafzai}
\address{Department of Basic Sciences and Humanities, National University of Sciences and Technology, Islamabad, Pakistan}
\email{yousafzaimath@gmail.com}
\author[M. U. I. Khan]{Murad ul Islam Khan*}\thanks{*Corresponding author}
\address{Department of Mathematics and Statistics, The University of Haripur, Pakistan}
\email{muradulislam@uoh.edu.pk}
\author[W. Sikandar]{Waseem Sikandar}
\address{Department of Mathematics and Statistics, The University of Haripur}
\email{waseem.sikandar@uoh.edu.pk}
\author[A. U. I. Khan]{Asad ul Islam Khan}
\address{Economics Department, Ibn Haldun University, Istanbul, Turkiye }
\email{asad.khan@ihu.edu.tr}

\begin{abstract}
	Metric dimensions and metric basis are graph invariants studied for their use in locating and indexing nodes in a graph. It was recently established that for bicyclic graph of type-III ($\Theta $-graphs), the metric dimension is $3$ only, when all paths have equal lengths, or when one of the outside path has a length $2$ more than the other two paths. In this article, we refute this claim and show that the case where the middle path is $2$ vertices more than the other two paths, also has metric dimension $3$. We also determine the metric dimension for other values of $p,q,r$ which were omitted in the recent research due to the constraint $p \leq q \leq r$. We also propose a graph-based technique to transform an agricultural supply chain logistics problem into a mathematical model, by using metric basis and metric dimensions. We provide a theoretical groundwork which can be used to model and solve these problems using machine learning algorithms.
\end{abstract}

\maketitle

\section{Introduction}
Graph Theory models real world relationships using entities, e.g., persons, websites, cities etc., as nodes/vertices and their relationships e.g., friendship, hyperlinks, roads etc., as links/edges. A graph $\G$ is called simple if any two connected vertices have only one edge, and if no vertex is connected to itself.

For a simple connected graph $\G=(\V,\E)$, $d(u,v)$ is called the distance between vertices $u,v \in \G$, and \[d(u,v)=\min\{|\P_{uv}| : P_{uv} \text{ is a path between $u$ and $v$}\}.\]
Here $|\P_{uv}|$ is called the length of the path and is defined to be the number of edges on the path $\P_{uv}$.

Let $\U=\{v_1,v_2, \cdots, v_t\} \subseteq \V(\G)$ be an ordered subset and $v \in \V(\G)$. Let $x \in \mathbb{Z}^t$ be defined as $x=(d(v,v_1), d(v,v_2), \cdots, d(v,v_t))$. This $x \in \mathbb{Z}^t$ is called the representation of $v$ with respect to $\U$, denoted by $\RR(v|\U)$. $\U$ is a "\textit{locating set}" \cite{sla2} or "\textit{resolving set}" \cite{CHARTRAND200099} if all vertices of $\G$ have distinct representation with respect to $\U$. Given $\mathcal{W}$ to be the family of all resolving sets of $\G$, then a minimal element of this family is called a metric basis. The cardinality of metric basis is called metric dimension of $\G$, denoted by $\beta(\G)$.

Metric dimension was introduced by Slater \cite{sla1} while studying the location of an intruder in a network. It was also independently studied by Hararay and Melter \cite{melter1976metric}. Since its inception, this idea has garnered a huge interest in scientific literature. Nazeer \textit{et} al. \cite{nazeer2021metric} worked with path graphs and calculated the metric dimension of many graphs arising from them. Carceres \textit{et} al. \cite{CACERES20122618} considered infinite graphs where all vertices have finite degree and provided a necessary condition for those graphs to have finite metric dimensions. They also characterized infinite trees with finite metric dimensions. Rezaei \textit{et} al. \cite{rezaei2022metric} determined the metric dimension of some generalized Cayley graphs. All unicyclic graphs with metric dimension $2$ were characterized in \cite{dudenko2017unicyclic}. Singh \textit{et} al. used the metric dimension parameter to identify some specific chemical structures \cite{singh2023vertex}. Sharma \textit{et} al. provided metric resolvability and topological characterization of some molecules in $H1N1$ antiviral drugs \cite{doi:10.1080/08927022.2023.2223718}. For further studies on resolving sets, metric basis and metric dimensions, one can refer to \cite{saenpholphat2004conditional, GENESON2022123, klein2012comparison, math10060962, bailey2010metric, caceres2007metric, yero2011metric} and the references therein.

Many different variants of the metric dimension have been defined and extensively studied, e.g., edge metric dimension \cite{kelenc2018uniquely, knor2021graphs, sharma2023edge}, local metric dimension \cite{okamoto2010local, fancy2021local, ghalavand2023conjecture}, and fault-tolerant metric dimension \cite{hernando2008fault, guo2020fault, sharma2022fault}.

The concepts of metric dimensions and metric basis have applications in many fields, for example, image processing \cite{MELTER1984113}, combinatorial optimization \cite{sebHo2004metric}, and pharmaceutical chemistry \cite{cameron1991designs}, just to name a few.

Metric dimension problem for unicyclic graphs has been studied by many authors. Eroh \textit{et} al. studied the relationship between metric dimension and zero forcing number, $\mathcal{Z}(\G)$, of trees and unicyclic graphs. Sedlar \textit{et} al \cite{sedlar2022vertex, sedlar2021bounds} established that the metric dimension of unicyclic graph is bounded and it can take only take two consecutive integers as its values. They also extended these results to leafless cacti graphs and showed that the bound for metric dimension for such graphs is $2c(\G)-1$.

Three different types of base bicyclic graphs were given in \cite{math11040869}, and the metric dimension for type-I and type-II base bicyclic graphs were calculated. On the other hand, it was established in \cite{10142411} that for $p \geq 2$, $\beta(\Theta_{p,p,p})=\beta(\Theta_{p,p,p+2})=3$, while $\beta(\Theta_{p,q,r})=2$, where $p \leq q \leq r$. Since $\Theta$-graphs are actually base bicyclic graphs of type-III, the results of $\Theta$-graphs are also applicable to base bicyclic graphs of type-III. We will further discuss these results in the motivation section of our article.

%Recently, it was established in \cite{sedlar2021bounds} that the metric dimension of unicyclic graph is bounded and it can take only take two consecutive integers as its values. Next, in \cite{sedlar2022vertex}, conditions were established under which these values are attained. These results were then extended to graphs with edge disjoint cycles, also known as \textit{cacti}, in \cite{SEDLAR2022126}. Since, the bounds in the class of cacti are dependent on the presence of \textit{leaves} (vertices of degree $1$), leafless cacti were investigated in \cite{SEDLAR2022127147} and it was established that the bound in this case is $2c(G)-1$. 

\subsection{Preliminaries}
A simple connected graph $\G=(\V,\E)$ is of order $n$ and size $m$, if it has $n$ vertices and $m$ edges. $\mathcal{\P}_n$, $\mathcal{\C}_n$ and $\mathcal{K}_n$, represent a path, cycle and complete graph of order $n$. By $\mathcal{K}_{n_1,n_2}$, we denote a bipartite graph with all possible edges having partitions of order $n_1$ and $n_2$. Disjoint union of graphs $\G$ and $\G'$ will be denoted by $\G \cup \G'$, while $\G+\G'$ will denote the graph constructed from $\G \cup \G'$ by joining all vertices of $\G$ and $\G'$.

\begin{Theorem}\cite{khuller1996landmarks,CHARTRAND200099} \label{basic}
	For a connected simple graph $\G=(\V,\E)$ having $2$ or more vertices, we have
	\begin{enumerate}[label=(\alph*)]
		\item $\beta(\G)=1$ iff $\G$ is an order $n$ path.
		\item $\beta(\G)=n-1$ iff $\G$ is an order $n$ complete graph.
		\item Let $\G$ be a cycle on $3$ or more vertices,  then $\beta(\G)=2$.
		\item For $n \geq 4$, $\beta(\G)=n-2$ iff $\G = \mathcal{K}_{n_1,n_2} (n_1,n_2 \geq 1, n_1+n_2=n), \G = \mathcal{K}_{n_1}+\xoverline{\mathcal{K}}_{n_2} (n_1 \geq 1, n_2 \geq 2, n_1+n_2=n)$, or $\G = \mathcal{K}_{n_1} + (\mathcal{K}_1 \cup \mathcal{K}_{n_2}) (n_1,n_2 \geq 1, n_1+n_2=n-1)$.
	\end{enumerate}
\end{Theorem}

\begin{Definition}
	A simple connected graph $\G$ is bicyclic, if it is of order $n$ and size $n+1$.
\end{Definition}

A vertex is said to be a leaf if it has degree $1$. A bicyclic graph is said to be a \textit{base bicyclic graph} if it does not have any leaves. There are three different types of base bicyclic graphs \cite{math11040869}:

\begin{itemize}
	\item $\C_{p,q}$ where cycles $\C_p$ and $\C_q$ share a single vertex.
	\item $\C_{p,r,q}$ where a vertex of $\C_p$ is connected to a vertex of $\C_q$ by a path of length $r$.
	\item $\C_{p,q,r}$ obtained by joining the end vertices of paths on $p$, $q$ and $r$ number of vertices.
\end{itemize}

Bicyclic graphs of type III, i.e., $\C_{p,q,r}$ are usually referred to as $\Theta$-graphs. They are studied extensively in the literature \cite{LIU201695, bukh2020turan, xu2024spectral} owing to their unique structure.

\subsection{Motivation}
Before proceeding further, let us introduce the notation of $\Theta$-graph used in \cite{10142411}.

\enquote{
	Therefore, let us introduce a necessary notation for $\Theta$-graphs. Let $\G$ be a $\Theta$-graph; by $u$ and $v$, we denote the two vertices of degree $3$ in $\G$. Notice that there are three distinct paths in $\G$ connecting $u$ and $v$, and we denote them by $\P_1=u_0u_1 \cdots u_p$, $\P_2=v_0v_1 \cdots v_q$, and $\P_3=w_0w_1 \cdots w_r$, so that $u_0=v_0=w_0=u$, $u_p=v_q=w_r=v$, and \underline{$p \leq q \leq r$}. The cycle in $\G$ induced by paths $\P_i$ and $\P_j$ will be denoted by $\C_{ij}$. A $\Theta$-graphs in which paths $\P_1$, $\P_2$, and $\P_3$ are of lengths $p,q,$ and $r$, respectively, is denoted by $\Theta_{p,q,r}$.}

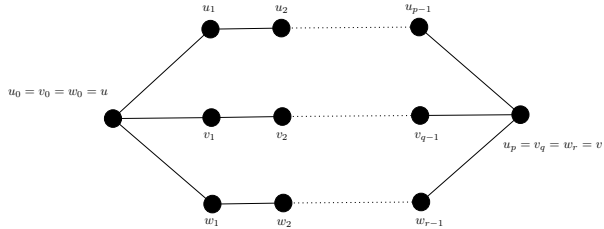
\begin{figure}[H]
	\centering
	
	\tikzset{every picture/.style={line width=0.75pt}} %set default line width to 0.75pt        
	\resizebox{8cm}{3cm}{%  
		\begin{tikzpicture}[x=0.75pt,y=0.75pt,yscale=-1,xscale=1]
			%uncomment if require: \path (0,300); %set diagram left start at 0, and has height of 300
			
			%Shape: Ellipse [id:dp26211884448221334] 
			\draw  [fill={rgb, 255:red, 0; green, 0; blue, 0 }  ,fill opacity=1 ] (217.33,220.03) .. controls (217.33,215.03) and (221.27,210.98) .. (226.12,210.98) .. controls (230.97,210.98) and (234.9,215.03) .. (234.9,220.03) .. controls (234.9,225.02) and (230.97,229.07) .. (226.12,229.07) .. controls (221.27,229.07) and (217.33,225.02) .. (217.33,220.03) -- cycle ;
			%Shape: Ellipse [id:dp1047201957912478] 
			\draw  [fill={rgb, 255:red, 0; green, 0; blue, 0 }  ,fill opacity=1 ] (290.11,218.94) .. controls (290.11,213.95) and (294.04,209.9) .. (298.89,209.9) .. controls (303.74,209.9) and (307.67,213.95) .. (307.67,218.94) .. controls (307.67,223.94) and (303.74,227.99) .. (298.89,227.99) .. controls (294.04,227.99) and (290.11,223.94) .. (290.11,218.94) -- cycle ;
			%Shape: Ellipse [id:dp4175039088738466] 
			\draw  [fill={rgb, 255:red, 0; green, 0; blue, 0 }  ,fill opacity=1 ] (431.44,217.86) .. controls (431.44,212.86) and (435.37,208.81) .. (440.22,208.81) .. controls (445.07,208.81) and (449,212.86) .. (449,217.86) .. controls (449,222.85) and (445.07,226.9) .. (440.22,226.9) .. controls (435.37,226.9) and (431.44,222.85) .. (431.44,217.86) -- cycle ;
			%Straight Lines [id:da42363136661863243] 
			\draw    (298.89,218.94) -- (226.12,220.03) ;
			%Straight Lines [id:da1875296369403272] 
			\draw  [dash pattern={on 0.84pt off 2.51pt}]  (440.22,217.86) -- (298.89,218.94) ;
			%Shape: Ellipse [id:dp7990433334058031] 
			\draw  [fill={rgb, 255:red, 0; green, 0; blue, 0 }  ,fill opacity=1 ] (216.33,131.03) .. controls (216.33,126.03) and (220.27,121.98) .. (225.12,121.98) .. controls (229.97,121.98) and (233.9,126.03) .. (233.9,131.03) .. controls (233.9,136.02) and (229.97,140.07) .. (225.12,140.07) .. controls (220.27,140.07) and (216.33,136.02) .. (216.33,131.03) -- cycle ;
			%Shape: Ellipse [id:dp22560949067737623] 
			\draw  [fill={rgb, 255:red, 0; green, 0; blue, 0 }  ,fill opacity=1 ] (289.11,129.94) .. controls (289.11,124.95) and (293.04,120.9) .. (297.89,120.9) .. controls (302.74,120.9) and (306.67,124.95) .. (306.67,129.94) .. controls (306.67,134.94) and (302.74,138.99) .. (297.89,138.99) .. controls (293.04,138.99) and (289.11,134.94) .. (289.11,129.94) -- cycle ;
			%Shape: Ellipse [id:dp647371503689542] 
			\draw  [fill={rgb, 255:red, 0; green, 0; blue, 0 }  ,fill opacity=1 ] (430.44,128.86) .. controls (430.44,123.86) and (434.37,119.81) .. (439.22,119.81) .. controls (444.07,119.81) and (448,123.86) .. (448,128.86) .. controls (448,133.85) and (444.07,137.9) .. (439.22,137.9) .. controls (434.37,137.9) and (430.44,133.85) .. (430.44,128.86) -- cycle ;
			%Straight Lines [id:da6484413564234863] 
			\draw    (297.89,129.94) -- (225.12,131.03) ;
			%Straight Lines [id:da7807329615543441] 
			\draw  [dash pattern={on 0.84pt off 2.51pt}]  (439.22,128.86) -- (297.89,129.94) ;
			%Shape: Ellipse [id:dp5487918478270626] 
			\draw  [fill={rgb, 255:red, 0; green, 0; blue, 0 }  ,fill opacity=1 ] (215.33,40.03) .. controls (215.33,35.03) and (219.27,30.98) .. (224.12,30.98) .. controls (228.97,30.98) and (232.9,35.03) .. (232.9,40.03) .. controls (232.9,45.02) and (228.97,49.07) .. (224.12,49.07) .. controls (219.27,49.07) and (215.33,45.02) .. (215.33,40.03) -- cycle ;
			%Shape: Ellipse [id:dp001806361777132226] 
			\draw  [fill={rgb, 255:red, 0; green, 0; blue, 0 }  ,fill opacity=1 ] (288.11,38.94) .. controls (288.11,33.95) and (292.04,29.9) .. (296.89,29.9) .. controls (301.74,29.9) and (305.67,33.95) .. (305.67,38.94) .. controls (305.67,43.94) and (301.74,47.99) .. (296.89,47.99) .. controls (292.04,47.99) and (288.11,43.94) .. (288.11,38.94) -- cycle ;
			%Shape: Ellipse [id:dp11155007010436657] 
			\draw  [fill={rgb, 255:red, 0; green, 0; blue, 0 }  ,fill opacity=1 ] (429.44,37.86) .. controls (429.44,32.86) and (433.37,28.81) .. (438.22,28.81) .. controls (443.07,28.81) and (447,32.86) .. (447,37.86) .. controls (447,42.85) and (443.07,46.9) .. (438.22,46.9) .. controls (433.37,46.9) and (429.44,42.85) .. (429.44,37.86) -- cycle ;
			%Straight Lines [id:da6664972778280742] 
			\draw    (296.89,38.94) -- (224.12,40.03) ;
			%Straight Lines [id:da898068467137938] 
			\draw  [dash pattern={on 0.84pt off 2.51pt}]  (438.22,37.86) -- (296.89,38.94) ;
			%Shape: Ellipse [id:dp3268001406972614] 
			\draw  [fill={rgb, 255:red, 0; green, 0; blue, 0 }  ,fill opacity=1 ] (115.33,132.03) .. controls (115.33,127.03) and (119.27,122.98) .. (124.12,122.98) .. controls (128.97,122.98) and (132.9,127.03) .. (132.9,132.03) .. controls (132.9,137.02) and (128.97,141.07) .. (124.12,141.07) .. controls (119.27,141.07) and (115.33,137.02) .. (115.33,132.03) -- cycle ;
			%Shape: Ellipse [id:dp15241186801060191] 
			\draw  [fill={rgb, 255:red, 0; green, 0; blue, 0 }  ,fill opacity=1 ] (533.33,128.03) .. controls (533.33,123.03) and (537.27,118.98) .. (542.12,118.98) .. controls (546.97,118.98) and (550.9,123.03) .. (550.9,128.03) .. controls (550.9,133.02) and (546.97,137.07) .. (542.12,137.07) .. controls (537.27,137.07) and (533.33,133.02) .. (533.33,128.03) -- cycle ;
			%Straight Lines [id:da8676642681219398] 
			\draw    (124.12,132.03) -- (226.12,220.03) ;
			%Straight Lines [id:da6437411907684805] 
			\draw    (438.22,37.86) -- (542.12,128.03) ;
			%Straight Lines [id:da5440728880309575] 
			\draw    (224.12,40.03) -- (124.12,132.03) ;
			%Straight Lines [id:da4549506118320621] 
			\draw    (542.12,128.03) -- (440.22,217.86) ;
			%Straight Lines [id:da33855317471187285] 
			\draw    (225.12,131.03) -- (124.12,132.03) ;
			%Straight Lines [id:da3825804658581089] 
			\draw    (542.12,128.03) -- (439.22,128.86) ;
			
			% Text Node
			\draw (216.46,233.63) node [anchor=north west][inner sep=0.75pt]    {$w_{1}$};
			% Text Node
			\draw (290.87,234.47) node [anchor=north west][inner sep=0.75pt]    {$w_{2}$};
			% Text Node
			\draw (432.92,233.63) node [anchor=north west][inner sep=0.75pt]    {$w_{r-1}$};
			% Text Node
			\draw (215.46,144.63) node [anchor=north west][inner sep=0.75pt]    {$v_{1}$};
			% Text Node
			\draw (289.87,144.47) node [anchor=north west][inner sep=0.75pt]    {$v_{2}$};
			% Text Node
			\draw (431.92,144.63) node [anchor=north west][inner sep=0.75pt]    {$v_{q-1}$};
			% Text Node
			\draw (214.46,13.63) node [anchor=north west][inner sep=0.75pt]    {$u_{1}$};
			% Text Node
			\draw (288.87,14.47) node [anchor=north west][inner sep=0.75pt]    {$u_{2}$};
			% Text Node
			\draw (422.27,13.47) node [anchor=north west][inner sep=0.75pt]    {$u_{p-1}$};
			% Text Node
			\draw (15,100.4) node [anchor=north west][inner sep=0.75pt]    {$u_{0} =v_{0} =w_{0} =u$};
			% Text Node
			\draw (523,155.4) node [anchor=north west][inner sep=0.75pt]    {$u_{p} =v_{q} =w_{r} =v$};

		\end{tikzpicture}
	}
	\caption{$\Theta$-graph with paths of lengths $p$, $q$, and $r$}
	\label{thetasedlar}
\end{figure}

Considering the definitions of bicyclic graphs of type-III and $\Theta$-graphs, we see that they are the same graphs. We now state a result from \cite{10142411} and provide a counterexample.

\begin{Theorem}{\cite[Theorem 2]{10142411}}\label{sedlarthm}
	\enquote{Let $\G$ be a $\Theta$-graph such that $\G \neq \Theta_{p,p,p}$ and $\Theta_{p,p,p+2}$ with $p \geq 2$. Then, $\beta(\G)=2$.}
\end{Theorem}

Let us consider the graph $\G$ given in the following figure.

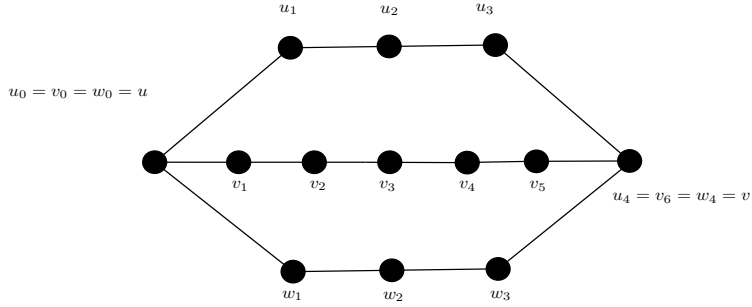
\begin{figure}[H]
	\centering
	\tikzset{every picture/.style={line width=0.75pt}} %set default line width to 0.75pt        
	\resizebox{10cm}{4cm}{%
		\begin{tikzpicture}[x=0.75pt,y=0.75pt,yscale=-1,xscale=1]
			%uncomment if require: \path (0,300); %set diagram left start at 0, and has height of 300
			
			%Shape: Ellipse [id:dp16191451903522203] 
			\draw  [fill={rgb, 255:red, 0; green, 0; blue, 0 }  ,fill opacity=1 ] (264.33,248.96) .. controls (264.33,243.96) and (268.27,239.91) .. (273.12,239.91) .. controls (277.97,239.91) and (281.9,243.96) .. (281.9,248.96) .. controls (281.9,253.95) and (277.97,258) .. (273.12,258) .. controls (268.27,258) and (264.33,253.95) .. (264.33,248.96) -- cycle ;
			%Shape: Ellipse [id:dp41902769302093623] 
			\draw  [fill={rgb, 255:red, 0; green, 0; blue, 0 }  ,fill opacity=1 ] (337.11,247.87) .. controls (337.11,242.88) and (341.04,238.83) .. (345.89,238.83) .. controls (350.74,238.83) and (354.67,242.88) .. (354.67,247.87) .. controls (354.67,252.87) and (350.74,256.92) .. (345.89,256.92) .. controls (341.04,256.92) and (337.11,252.87) .. (337.11,247.87) -- cycle ;
			%Shape: Ellipse [id:dp22618931337268] 
			\draw  [fill={rgb, 255:red, 0; green, 0; blue, 0 }  ,fill opacity=1 ] (415.44,246.79) .. controls (415.44,241.79) and (419.37,237.74) .. (424.22,237.74) .. controls (429.07,237.74) and (433,241.79) .. (433,246.79) .. controls (433,251.78) and (429.07,255.83) .. (424.22,255.83) .. controls (419.37,255.83) and (415.44,251.78) .. (415.44,246.79) -- cycle ;
			%Straight Lines [id:da8862607107789615] 
			\draw    (345.89,247.87) -- (273.12,248.96) ;
			%Shape: Ellipse [id:dp09377870104797847] 
			\draw  [fill={rgb, 255:red, 0; green, 0; blue, 0 }  ,fill opacity=1 ] (224.22,161) .. controls (224.22,156) and (228.15,151.95) .. (233,151.95) .. controls (237.85,151.95) and (241.78,156) .. (241.78,161) .. controls (241.78,166) and (237.85,170.05) .. (233,170.05) .. controls (228.15,170.05) and (224.22,166) .. (224.22,161) -- cycle ;
			%Shape: Ellipse [id:dp5081861582441105] 
			\draw  [fill={rgb, 255:red, 0; green, 0; blue, 0 }  ,fill opacity=1 ] (280,161) .. controls (280,156) and (283.93,151.95) .. (288.78,151.95) .. controls (293.63,151.95) and (297.56,156) .. (297.56,161) .. controls (297.56,166) and (293.63,170.05) .. (288.78,170.05) .. controls (283.93,170.05) and (280,166) .. (280,161) -- cycle ;
			%Shape: Ellipse [id:dp43400477851400643] 
			\draw  [fill={rgb, 255:red, 0; green, 0; blue, 0 }  ,fill opacity=1 ] (443.65,160.29) .. controls (443.65,155.29) and (447.58,151.24) .. (452.43,151.24) .. controls (457.28,151.24) and (461.21,155.29) .. (461.21,160.29) .. controls (461.21,165.28) and (457.28,169.33) .. (452.43,169.33) .. controls (447.58,169.33) and (443.65,165.28) .. (443.65,160.29) -- cycle ;
			%Straight Lines [id:da57487851295782] 
			\draw    (279.78,161) -- (233,161) ;
			%Shape: Ellipse [id:dp5376228641544838] 
			\draw  [fill={rgb, 255:red, 0; green, 0; blue, 0 }  ,fill opacity=1 ] (262.33,68.96) .. controls (262.33,63.96) and (266.27,59.91) .. (271.12,59.91) .. controls (275.97,59.91) and (279.9,63.96) .. (279.9,68.96) .. controls (279.9,73.95) and (275.97,78) .. (271.12,78) .. controls (266.27,78) and (262.33,73.95) .. (262.33,68.96) -- cycle ;
			%Shape: Ellipse [id:dp5942714822753774] 
			\draw  [fill={rgb, 255:red, 0; green, 0; blue, 0 }  ,fill opacity=1 ] (335.11,67.87) .. controls (335.11,62.88) and (339.04,58.83) .. (343.89,58.83) .. controls (348.74,58.83) and (352.67,62.88) .. (352.67,67.87) .. controls (352.67,72.87) and (348.74,76.92) .. (343.89,76.92) .. controls (339.04,76.92) and (335.11,72.87) .. (335.11,67.87) -- cycle ;
			%Shape: Ellipse [id:dp657035198850213] 
			\draw  [fill={rgb, 255:red, 0; green, 0; blue, 0 }  ,fill opacity=1 ] (413.44,66.79) .. controls (413.44,61.79) and (417.37,57.74) .. (422.22,57.74) .. controls (427.07,57.74) and (431,61.79) .. (431,66.79) .. controls (431,71.78) and (427.07,75.83) .. (422.22,75.83) .. controls (417.37,75.83) and (413.44,71.78) .. (413.44,66.79) -- cycle ;
			%Straight Lines [id:da8477930916009733] 
			\draw    (343.89,67.87) -- (271.12,68.96) ;
			%Shape: Ellipse [id:dp6833863108365512] 
			\draw  [fill={rgb, 255:red, 0; green, 0; blue, 0 }  ,fill opacity=1 ] (162.33,160.96) .. controls (162.33,155.96) and (166.27,151.91) .. (171.12,151.91) .. controls (175.97,151.91) and (179.9,155.96) .. (179.9,160.96) .. controls (179.9,165.95) and (175.97,170) .. (171.12,170) .. controls (166.27,170) and (162.33,165.95) .. (162.33,160.96) -- cycle ;
			%Shape: Ellipse [id:dp6305830285842267] 
			\draw  [fill={rgb, 255:red, 0; green, 0; blue, 0 }  ,fill opacity=1 ] (512.33,159.96) .. controls (512.33,154.96) and (516.27,150.91) .. (521.12,150.91) .. controls (525.97,150.91) and (529.9,154.96) .. (529.9,159.96) .. controls (529.9,164.95) and (525.97,169) .. (521.12,169) .. controls (516.27,169) and (512.33,164.95) .. (512.33,159.96) -- cycle ;
			%Straight Lines [id:da729356254059131] 
			\draw    (171.12,160.96) -- (273.12,248.96) ;
			%Straight Lines [id:da24413337076384267] 
			\draw    (422.22,66.79) -- (521.12,159.96) ;
			%Straight Lines [id:da9411866014002728] 
			\draw    (271.12,68.96) -- (171.12,160.96) ;
			%Straight Lines [id:da9905514955075014] 
			\draw    (526.12,156.96) -- (424.22,246.79) ;
			%Straight Lines [id:da959281441271997] 
			\draw    (233,161) -- (171.12,160.96) ;
			%Straight Lines [id:da6624822177106247] 
			\draw    (521.12,159.96) -- (446.43,160.29) ;
			%Straight Lines [id:da8101412744406227] 
			\draw    (416.66,66.79) -- (343.89,67.87) ;
			%Straight Lines [id:da7305346169968814] 
			\draw    (418.66,246.79) -- (345.89,247.87) ;
			%Shape: Ellipse [id:dp16647256944623012] 
			\draw  [fill={rgb, 255:red, 0; green, 0; blue, 0 }  ,fill opacity=1 ] (335.56,161) .. controls (335.56,156) and (339.5,151.95) .. (344.35,151.95) .. controls (349.2,151.95) and (353.13,156) .. (353.13,161) .. controls (353.13,166) and (349.2,170.05) .. (344.35,170.05) .. controls (339.5,170.05) and (335.56,166) .. (335.56,161) -- cycle ;
			%Straight Lines [id:da07753813512504548] 
			\draw    (326.56,161) -- (279.78,161) ;
			%Straight Lines [id:da11572452531501409] 
			\draw    (395.43,161.29) -- (326.56,161) ;
			%Shape: Ellipse [id:dp7020360317039567] 
			\draw  [fill={rgb, 255:red, 0; green, 0; blue, 0 }  ,fill opacity=1 ] (392.65,161.29) .. controls (392.65,156.29) and (396.58,152.24) .. (401.43,152.24) .. controls (406.28,152.24) and (410.21,156.29) .. (410.21,161.29) .. controls (410.21,166.28) and (406.28,170.33) .. (401.43,170.33) .. controls (396.58,170.33) and (392.65,166.28) .. (392.65,161.29) -- cycle ;
			%Straight Lines [id:da6989099598486703] 
			\draw    (446.43,160.29) -- (395.43,161.29) ;
			
			% Text Node
			\draw (263.46,262.56) node [anchor=north west][inner sep=0.75pt]    {$w_{1}$};
			% Text Node
			\draw (337.87,263.4) node [anchor=north west][inner sep=0.75pt]    {$w_{2}$};
			% Text Node
			\draw (416.92,262.56) node [anchor=north west][inner sep=0.75pt]    {$w_{3}$};
			% Text Node
			\draw (226.46,174.56) node [anchor=north west][inner sep=0.75pt]    {$v_{1}$};
			% Text Node
			\draw (283.87,174.4) node [anchor=north west][inner sep=0.75pt]    {$v_{2}$};
			% Text Node
			\draw (334.92,174.56) node [anchor=north west][inner sep=0.75pt]    {$v_{3}$};
			% Text Node
			\draw (261.46,32.56) node [anchor=north west][inner sep=0.75pt]    {$u_{1}$};
			% Text Node
			\draw (335.87,33.4) node [anchor=north west][inner sep=0.75pt]    {$u_{2}$};
			% Text Node
			\draw (406.27,32.4) node [anchor=north west][inner sep=0.75pt]    {$u_{3}$};
			% Text Node
			\draw (62,100.33) node [anchor=north west][inner sep=0.75pt]    {$u_{0} =v_{0} =w_{0} =u$};
			% Text Node
			\draw (507,184.33) node [anchor=north west][inner sep=0.75pt]    {$u_{4} =v_{6} =w_{4} =v$};
			% Text Node
			\draw (393.92,174.56) node [anchor=north west][inner sep=0.75pt]    {$v_{4}$};
			% Text Node
			\draw (445.92,174.56) node [anchor=north west][inner sep=0.75pt]    {$v_{5}$};

		\end{tikzpicture}
	}
	\caption{$\Theta_{4,6,4}$-graph}
	\label{theta464}
\end{figure}

Since the paths $\P_1, \P_2,$ and $\P_3$ are of lengths $4, 6, $ and $4$ respectively, we see that $\G \simeq \Theta_{4,6,4}$. Again, since $\G \neq \Theta_{p,p,p}$ and $\G\neq \Theta_{p,p,p+2}$ with $p \geq 2$, by Theorem \ref{sedlarthm}, $\beta(\G)=\beta(\Theta_{4,6,4})=2$. A very basic computer program can verify that no resolving set of cardinality $2$ exists for this graph.

On the other hand, if we consider the set $\W=\{u_1,u_2,v_2\} \subset \V(\G)$, it is an easy task to show
that the set $\W$ is indeed a resolving set for $\G$. Moreover, if we remove any vertex from $\W$, the remaining set is not a resolving set, e.g., removing $u_1$, we see that the vertices $v_5$ and $w_1$ are not resolved by $\W-u_1$. A similar argument is applicable to $u_2$ and $v_2$. This ensures that $\W$ is a minimal resolving set of $\G$ and $\beta(\G)=3$.

This example is at odds with Theorem \ref{sedlarthm}. Using this example as a base, we conjecture that $\beta(\Theta_{p,p+2,p})=3$. We will prove this claim later on.

Returning to the notation defined for $\Theta$-graphs, it was specified that $p \leq q \leq r$. This restriction on the values of $p,q$ and $r$ eliminates whole classes of $\Theta$-graphs from the discussion. These are given in the following.

\begin{itemize}
	\item $\Theta$-graphs where $p \leq q$ but $q \nleq r$, then obviously, $q >r$. We can subdivide it into two possibilities, $q \geq p > r$, and $q > r \geq p$.
	\item $\Theta$-graphs for which $\P_1$ and $\P_3$ are of equal length $(p=r)$ and $\P_2$ is unrestricted ($q$ may be less, greater or equal to $p$).
\end{itemize}

Although some classes are missing in the proofs of \cite{10142411}, a great many of the classes have already been solved there. For the sake of completeness, we consider all classes of $\Theta$-graphs in this study. Our proofs not only serve as a verification for their results but also fill in the gaps present in that study.

\subsection{Notations}

For the purpose of this article, we will be using the following notation for $\Theta$-graphs, henceforth labelled as \textit{bicyclic graph of type-III}.

Let $\C_{p,q,r}$ be a graph constructed from three distinct paths $\P_1$ , $\P_2$ and $\P_3$ having vertices $p, q,$ and $r$ respectively. Starting and ending vertices of $\P_1$ and $\P_3$ are then connected to the starting and ending vertices of $\P_2$, respectively. If we denote the vertices as $v_1,v_2, \cdots, v_{p+q+r}$, then this graph is given in the following figure. 

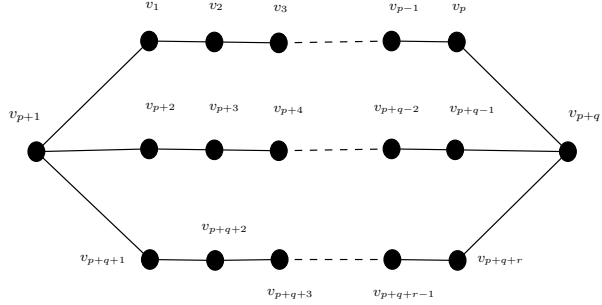
\begin{figure}[H]
	\centering
	\tikzset{every picture/.style={line width=0.75pt}} %set default line width to 0.75pt        
	\resizebox{8cm}{4cm}{%
		
		\begin{tikzpicture}[x=0.75pt,y=0.75pt,yscale=-1,xscale=1]
			%uncomment if require: \path (0,300); %set diagram left start at 0, and has height of 300
			
			%Shape: Ellipse [id:dp16635427090155708] 
			\draw  [fill={rgb, 255:red, 0; green, 0; blue, 0 }  ,fill opacity=1 ] (202.63,51.49) .. controls (202.63,46.92) and (205.6,43.21) .. (209.26,43.21) .. controls (212.92,43.21) and (215.89,46.92) .. (215.89,51.49) .. controls (215.89,56.07) and (212.92,59.77) .. (209.26,59.77) .. controls (205.6,59.77) and (202.63,56.07) .. (202.63,51.49) -- cycle ;
			%Shape: Ellipse [id:dp8802797495475787] 
			\draw  [fill={rgb, 255:red, 0; green, 0; blue, 0 }  ,fill opacity=1 ] (254.56,52.27) .. controls (254.56,47.69) and (257.52,43.99) .. (261.19,43.99) .. controls (264.85,43.99) and (267.82,47.69) .. (267.82,52.27) .. controls (267.82,56.84) and (264.85,60.55) .. (261.19,60.55) .. controls (257.52,60.55) and (254.56,56.84) .. (254.56,52.27) -- cycle ;
			%Shape: Ellipse [id:dp3846254394033439] 
			\draw  [fill={rgb, 255:red, 0; green, 0; blue, 0 }  ,fill opacity=1 ] (305.84,53.04) .. controls (305.84,48.46) and (308.81,44.76) .. (312.47,44.76) .. controls (316.14,44.76) and (319.1,48.46) .. (319.1,53.04) .. controls (319.1,57.61) and (316.14,61.32) .. (312.47,61.32) .. controls (308.81,61.32) and (305.84,57.61) .. (305.84,53.04) -- cycle ;
			%Shape: Ellipse [id:dp8282056851105979] 
			\draw  [fill={rgb, 255:red, 0; green, 0; blue, 0 }  ,fill opacity=1 ] (395.6,51.49) .. controls (395.6,46.92) and (398.56,43.21) .. (402.23,43.21) .. controls (405.89,43.21) and (408.86,46.92) .. (408.86,51.49) .. controls (408.86,56.07) and (405.89,59.77) .. (402.23,59.77) .. controls (398.56,59.77) and (395.6,56.07) .. (395.6,51.49) -- cycle ;
			%Shape: Ellipse [id:dp9070970416713975] 
			\draw  [fill={rgb, 255:red, 0; green, 0; blue, 0 }  ,fill opacity=1 ] (447.52,52.27) .. controls (447.52,47.69) and (450.49,43.99) .. (454.15,43.99) .. controls (457.82,43.99) and (460.78,47.69) .. (460.78,52.27) .. controls (460.78,56.84) and (457.82,60.55) .. (454.15,60.55) .. controls (450.49,60.55) and (447.52,56.84) .. (447.52,52.27) -- cycle ;
			%Straight Lines [id:da8736722249593427] 
			\draw    (209.26,51.49) -- (261.19,52.27) ;
			%Straight Lines [id:da3293275737875725] 
			\draw    (261.19,52.27) -- (312.47,53.04) ;
			%Straight Lines [id:da3018374723956687] 
			\draw    (402.23,51.49) -- (454.15,52.27) ;
			%Straight Lines [id:da12412569017877817] 
			\draw  [dash pattern={on 4.5pt off 4.5pt}]  (313.11,53.04) -- (402.23,51.49) ;
			%Shape: Ellipse [id:dp18426623253223107] 
			\draw  [fill={rgb, 255:red, 0; green, 0; blue, 0 }  ,fill opacity=1 ] (202.63,143.32) .. controls (202.63,138.74) and (205.6,135.03) .. (209.26,135.03) .. controls (212.92,135.03) and (215.89,138.74) .. (215.89,143.32) .. controls (215.89,147.89) and (212.92,151.6) .. (209.26,151.6) .. controls (205.6,151.6) and (202.63,147.89) .. (202.63,143.32) -- cycle ;
			%Shape: Ellipse [id:dp3700294851992061] 
			\draw  [fill={rgb, 255:red, 0; green, 0; blue, 0 }  ,fill opacity=1 ] (254.56,144.09) .. controls (254.56,139.51) and (257.52,135.81) .. (261.19,135.81) .. controls (264.85,135.81) and (267.82,139.51) .. (267.82,144.09) .. controls (267.82,148.66) and (264.85,152.37) .. (261.19,152.37) .. controls (257.52,152.37) and (254.56,148.66) .. (254.56,144.09) -- cycle ;
			%Shape: Ellipse [id:dp7253037431435962] 
			\draw  [fill={rgb, 255:red, 0; green, 0; blue, 0 }  ,fill opacity=1 ] (305.84,144.86) .. controls (305.84,140.29) and (308.81,136.58) .. (312.47,136.58) .. controls (316.14,136.58) and (319.1,140.29) .. (319.1,144.86) .. controls (319.1,149.43) and (316.14,153.14) .. (312.47,153.14) .. controls (308.81,153.14) and (305.84,149.43) .. (305.84,144.86) -- cycle ;
			%Shape: Ellipse [id:dp2540205298297442] 
			\draw  [fill={rgb, 255:red, 0; green, 0; blue, 0 }  ,fill opacity=1 ] (395.6,143.32) .. controls (395.6,138.74) and (398.56,135.03) .. (402.23,135.03) .. controls (405.89,135.03) and (408.86,138.74) .. (408.86,143.32) .. controls (408.86,147.89) and (405.89,151.6) .. (402.23,151.6) .. controls (398.56,151.6) and (395.6,147.89) .. (395.6,143.32) -- cycle ;
			%Shape: Ellipse [id:dp38088877181040215] 
			\draw  [fill={rgb, 255:red, 0; green, 0; blue, 0 }  ,fill opacity=1 ] (446.24,144.09) .. controls (446.24,139.51) and (449.21,135.81) .. (452.87,135.81) .. controls (456.53,135.81) and (459.5,139.51) .. (459.5,144.09) .. controls (459.5,148.66) and (456.53,152.37) .. (452.87,152.37) .. controls (449.21,152.37) and (446.24,148.66) .. (446.24,144.09) -- cycle ;
			%Straight Lines [id:da4301962026653905] 
			\draw    (209.26,143.32) -- (261.19,144.09) ;
			%Straight Lines [id:da267535483401907] 
			\draw    (261.19,144.09) -- (312.47,144.86) ;
			%Straight Lines [id:da4002191182966446] 
			\draw    (402.23,143.32) -- (452.87,144.09) ;
			%Straight Lines [id:da4889568533501578] 
			\draw  [dash pattern={on 4.5pt off 4.5pt}]  (313.11,144.86) -- (402.23,143.32) ;
			%Shape: Ellipse [id:dp7891577231502005] 
			\draw  [fill={rgb, 255:red, 0; green, 0; blue, 0 }  ,fill opacity=1 ] (203.27,237.45) .. controls (203.27,232.88) and (206.24,229.17) .. (209.9,229.17) .. controls (213.56,229.17) and (216.53,232.88) .. (216.53,237.45) .. controls (216.53,242.02) and (213.56,245.73) .. (209.9,245.73) .. controls (206.24,245.73) and (203.27,242.02) .. (203.27,237.45) -- cycle ;
			%Shape: Ellipse [id:dp2915666142254518] 
			\draw  [fill={rgb, 255:red, 0; green, 0; blue, 0 }  ,fill opacity=1 ] (255.2,238.22) .. controls (255.2,233.65) and (258.17,229.94) .. (261.83,229.94) .. controls (265.49,229.94) and (268.46,233.65) .. (268.46,238.22) .. controls (268.46,242.8) and (265.49,246.5) .. (261.83,246.5) .. controls (258.17,246.5) and (255.2,242.8) .. (255.2,238.22) -- cycle ;
			%Shape: Ellipse [id:dp46870374277675464] 
			\draw  [fill={rgb, 255:red, 0; green, 0; blue, 0 }  ,fill opacity=1 ] (306.48,237.45) .. controls (306.48,232.88) and (309.45,229.17) .. (313.11,229.17) .. controls (316.78,229.17) and (319.75,232.88) .. (319.75,237.45) .. controls (319.75,242.02) and (316.78,245.73) .. (313.11,245.73) .. controls (309.45,245.73) and (306.48,242.02) .. (306.48,237.45) -- cycle ;
			%Shape: Ellipse [id:dp47039762381439054] 
			\draw  [fill={rgb, 255:red, 0; green, 0; blue, 0 }  ,fill opacity=1 ] (396.24,237.45) .. controls (396.24,232.88) and (399.21,229.17) .. (402.87,229.17) .. controls (406.53,229.17) and (409.5,232.88) .. (409.5,237.45) .. controls (409.5,242.02) and (406.53,245.73) .. (402.87,245.73) .. controls (399.21,245.73) and (396.24,242.02) .. (396.24,237.45) -- cycle ;
			%Shape: Ellipse [id:dp12922289153319544] 
			\draw  [fill={rgb, 255:red, 0; green, 0; blue, 0 }  ,fill opacity=1 ] (448.17,238.22) .. controls (448.17,233.65) and (451.13,229.94) .. (454.8,229.94) .. controls (458.46,229.94) and (461.43,233.65) .. (461.43,238.22) .. controls (461.43,242.8) and (458.46,246.5) .. (454.8,246.5) .. controls (451.13,246.5) and (448.17,242.8) .. (448.17,238.22) -- cycle ;
			%Straight Lines [id:da9291545993801622] 
			\draw    (209.9,237.45) -- (261.83,238.22) ;
			%Straight Lines [id:da752735262960353] 
			\draw    (261.83,238.22) -- (313.11,237.45) ;
			%Straight Lines [id:da1137638317630345] 
			\draw    (402.87,237.45) -- (454.8,238.22) ;
			%Straight Lines [id:da8729478133992548] 
			\draw  [dash pattern={on 4.5pt off 4.5pt}]  (313.11,237.45) -- (402.87,237.45) ;
			%Shape: Ellipse [id:dp7986787253571883] 
			\draw  [fill={rgb, 255:red, 0; green, 0; blue, 0 }  ,fill opacity=1 ] (535.99,145.63) .. controls (535.99,141.06) and (538.96,137.35) .. (542.63,137.35) .. controls (546.29,137.35) and (549.26,141.06) .. (549.26,145.63) .. controls (549.26,150.2) and (546.29,153.91) .. (542.63,153.91) .. controls (538.96,153.91) and (535.99,150.2) .. (535.99,145.63) -- cycle ;
			%Shape: Ellipse [id:dp3024124106838686] 
			\draw  [fill={rgb, 255:red, 0; green, 0; blue, 0 }  ,fill opacity=1 ] (112.87,145.63) .. controls (112.87,141.06) and (115.84,137.35) .. (119.51,137.35) .. controls (123.17,137.35) and (126.14,141.06) .. (126.14,145.63) .. controls (126.14,150.2) and (123.17,153.91) .. (119.51,153.91) .. controls (115.84,153.91) and (112.87,150.2) .. (112.87,145.63) -- cycle ;
			%Straight Lines [id:da17283790954788092] 
			\draw    (119.51,145.63) -- (209.9,237.45) ;
			%Straight Lines [id:da978542803751318] 
			\draw    (454.15,52.27) -- (542.63,145.63) ;
			%Straight Lines [id:da3110620361542831] 
			\draw    (209.26,51.49) -- (119.51,145.63) ;
			%Straight Lines [id:da24570502556348273] 
			\draw    (542.63,145.63) -- (454.8,238.22) ;
			%Straight Lines [id:da7019535169702733] 
			\draw    (119.51,145.63) -- (209.26,143.32) ;
			%Straight Lines [id:da9516218219812953] 
			\draw    (452.87,144.09) -- (542.63,145.63) ;
			
			% Text Node
			\draw (204.92,16.73) node [anchor=north west][inner sep=0.75pt]  [font=\small]  {$v_{1}$};
			% Text Node
			\draw (255.57,17.5) node [anchor=north west][inner sep=0.75pt]  [font=\small]  {$v_{2}$};
			% Text Node
			\draw (306.85,19.04) node [anchor=north west][inner sep=0.75pt]  [font=\small]  {$v_{3}$};
			% Text Node
			\draw (398.94,18.27) node [anchor=north west][inner sep=0.75pt]  [font=\small]  {$v_{p-1}$};
			% Text Node
			\draw (448.95,18.36) node [anchor=north west][inner sep=0.75pt]  [font=\small]  {$v_{p}$};
			% Text Node
			\draw (204.78,102.55) node [anchor=north west][inner sep=0.75pt]  [font=\small]  {$v_{p+2}$};
			% Text Node
			\draw (255.43,103.32) node [anchor=north west][inner sep=0.75pt]  [font=\small]  {$v_{p+3}$};
			% Text Node
			\draw (306.71,104.86) node [anchor=north west][inner sep=0.75pt]  [font=\small]  {$v_{p+4}$};
			% Text Node
			\draw (386.52,104.09) node [anchor=north west][inner sep=0.75pt]  [font=\small]  {$v_{p+q-2}$};
			% Text Node
			\draw (446.83,105.18) node [anchor=north west][inner sep=0.75pt]  [font=\small]  {$v_{p+q-1}$};
			% Text Node
			\draw (152.86,232.04) node [anchor=north west][inner sep=0.75pt]  [font=\small]  {$v_{p+q+1}$};
			% Text Node
			\draw (249.5,205.81) node [anchor=north west][inner sep=0.75pt]  [font=\small]  {$v_{p+q+2}$};
			% Text Node
			\draw (301.79,261.36) node [anchor=north west][inner sep=0.75pt]  [font=\small]  {$v_{p+q+3}$};
			% Text Node
			\draw (386.16,261.58) node [anchor=north west][inner sep=0.75pt]  [font=\small]  {$v_{p+q+r-1}$};
			% Text Node
			\draw (469.16,232.67) node [anchor=north west][inner sep=0.75pt]  [font=\small]  {$v_{p+q+r}$};
			% Text Node
			\draw (95.97,109.95) node [anchor=north west][inner sep=0.75pt]    {$v_{p+1}$};
			% Text Node
			\draw (541.51,108.87) node [anchor=north west][inner sep=0.75pt]    {$v_{p+q}$};

		\end{tikzpicture}
		
	}
	\caption{Type-III base bicyclic graph}
	\label{typeiiibicyclic}
\end{figure}	

It can easily be observed that $\C_{p,q,r} \simeq \Theta_{p+1,q-1,r+1}$. This also shows that the paths $\P_1$, $\P_2$ and $\P_3$ are of lengths $p+1$, $q-1$ and $r+1$, respectively.

\section{Results on Bicyclic Graphs Type-III:}

It is emphasized that interchanging the values of $p$ and $r$ in $\C_{p,q,r}$, we get an isomorphic graph, i.e., $\C_{p,q,r} \simeq \C_{r,q,p}$ for fixed $p,r$. This concept is explained in the following figure.

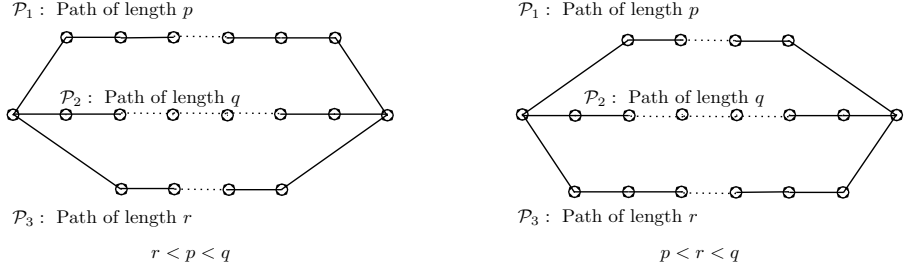
\begin{figure}[H]
	\centering
	\tikzset{every picture/.style={line width=0.75pt}} %set default line width to 0.75pt        
	\resizebox{12cm}{3.5cm}{%
		\begin{tikzpicture}[x=0.75pt,y=0.75pt,yscale=-1,xscale=1]
			%uncomment if require: \path (0,300); %set diagram left start at 0, and has height of 300
			
			%Shape: Ellipse [id:dp21404551466603983] 
			\draw  [fill={rgb, 255:red, 0; green, 0; blue, 0 }  ,fill opacity=1 ][line width=6]  (42.81,119.94) .. controls (42.81,119.7) and (42.99,119.5) .. (43.21,119.5) .. controls (43.44,119.5) and (43.62,119.7) .. (43.62,119.94) .. controls (43.62,120.18) and (43.44,120.37) .. (43.21,120.37) .. controls (42.99,120.37) and (42.81,120.18) .. (42.81,119.94) -- cycle ;
			%Shape: Ellipse [id:dp7250024208717187] 
			\draw  [fill={rgb, 255:red, 0; green, 0; blue, 0 }  ,fill opacity=1 ][line width=6]  (79.75,119.94) .. controls (79.75,119.7) and (79.93,119.5) .. (80.16,119.5) .. controls (80.38,119.5) and (80.56,119.7) .. (80.56,119.94) .. controls (80.56,120.18) and (80.38,120.37) .. (80.16,120.37) .. controls (79.93,120.37) and (79.75,120.18) .. (79.75,119.94) -- cycle ;
			%Shape: Ellipse [id:dp21535716554647588] 
			\draw  [fill={rgb, 255:red, 0; green, 0; blue, 0 }  ,fill opacity=1 ][line width=6]  (117.44,119.94) .. controls (117.44,119.7) and (117.62,119.5) .. (117.85,119.5) .. controls (118.07,119.5) and (118.25,119.7) .. (118.25,119.94) .. controls (118.25,120.18) and (118.07,120.37) .. (117.85,120.37) .. controls (117.62,120.37) and (117.44,120.18) .. (117.44,119.94) -- cycle ;
			%Shape: Ellipse [id:dp37134508763709007] 
			\draw  [fill={rgb, 255:red, 0; green, 0; blue, 0 }  ,fill opacity=1 ][line width=6]  (154.39,119.94) .. controls (154.39,119.7) and (154.57,119.5) .. (154.79,119.5) .. controls (155.01,119.5) and (155.19,119.7) .. (155.19,119.94) .. controls (155.19,120.18) and (155.01,120.37) .. (154.79,120.37) .. controls (154.57,120.37) and (154.39,120.18) .. (154.39,119.94) -- cycle ;
			%Shape: Ellipse [id:dp9501495274412106] 
			\draw  [fill={rgb, 255:red, 0; green, 0; blue, 0 }  ,fill opacity=1 ][line width=6]  (192.45,120.34) .. controls (192.45,120.1) and (192.63,119.9) .. (192.85,119.9) .. controls (193.07,119.9) and (193.26,120.1) .. (193.26,120.34) .. controls (193.26,120.58) and (193.07,120.77) .. (192.85,120.77) .. controls (192.63,120.77) and (192.45,120.58) .. (192.45,120.34) -- cycle ;
			%Shape: Ellipse [id:dp5764550201588585] 
			\draw  [fill={rgb, 255:red, 0; green, 0; blue, 0 }  ,fill opacity=1 ][line width=6]  (229.39,119.54) .. controls (229.39,119.3) and (229.57,119.1) .. (229.8,119.1) .. controls (230.02,119.1) and (230.2,119.3) .. (230.2,119.54) .. controls (230.2,119.77) and (230.02,119.97) .. (229.8,119.97) .. controls (229.57,119.97) and (229.39,119.77) .. (229.39,119.54) -- cycle ;
			%Shape: Ellipse [id:dp07093232242709147] 
			\draw  [fill={rgb, 255:red, 0; green, 0; blue, 0 }  ,fill opacity=1 ][line width=6]  (267.08,119.54) .. controls (267.08,119.3) and (267.26,119.1) .. (267.49,119.1) .. controls (267.71,119.1) and (267.89,119.3) .. (267.89,119.54) .. controls (267.89,119.77) and (267.71,119.97) .. (267.49,119.97) .. controls (267.26,119.97) and (267.08,119.77) .. (267.08,119.54) -- cycle ;
			%Shape: Ellipse [id:dp9746748768922846] 
			\draw  [fill={rgb, 255:red, 0; green, 0; blue, 0 }  ,fill opacity=1 ][line width=6]  (304.03,120.34) .. controls (304.03,120.1) and (304.21,119.9) .. (304.43,119.9) .. controls (304.65,119.9) and (304.83,120.1) .. (304.83,120.34) .. controls (304.83,120.58) and (304.65,120.77) .. (304.43,120.77) .. controls (304.21,120.77) and (304.03,120.58) .. (304.03,120.34) -- cycle ;
			%Shape: Ellipse [id:dp37039164679023195] 
			\draw  [fill={rgb, 255:red, 0; green, 0; blue, 0 }  ,fill opacity=1 ][line width=6]  (80.25,63.84) .. controls (80.25,63.6) and (80.43,63.41) .. (80.65,63.41) .. controls (80.88,63.41) and (81.06,63.6) .. (81.06,63.84) .. controls (81.06,64.08) and (80.88,64.27) .. (80.65,64.27) .. controls (80.43,64.27) and (80.25,64.08) .. (80.25,63.84) -- cycle ;
			%Shape: Ellipse [id:dp33080494485207135] 
			\draw  [fill={rgb, 255:red, 0; green, 0; blue, 0 }  ,fill opacity=1 ][line width=6]  (117.94,63.84) .. controls (117.94,63.6) and (118.12,63.41) .. (118.34,63.41) .. controls (118.57,63.41) and (118.75,63.6) .. (118.75,63.84) .. controls (118.75,64.08) and (118.57,64.27) .. (118.34,64.27) .. controls (118.12,64.27) and (117.94,64.08) .. (117.94,63.84) -- cycle ;
			%Shape: Ellipse [id:dp865976791739713] 
			\draw  [fill={rgb, 255:red, 0; green, 0; blue, 0 }  ,fill opacity=1 ][line width=6]  (154.88,63.84) .. controls (154.88,63.6) and (155.06,63.41) .. (155.29,63.41) .. controls (155.51,63.41) and (155.69,63.6) .. (155.69,63.84) .. controls (155.69,64.08) and (155.51,64.27) .. (155.29,64.27) .. controls (155.06,64.27) and (154.88,64.08) .. (154.88,63.84) -- cycle ;
			%Shape: Ellipse [id:dp25247952795551853] 
			\draw  [fill={rgb, 255:red, 0; green, 0; blue, 0 }  ,fill opacity=1 ][line width=6]  (192.95,64.24) .. controls (192.95,64) and (193.13,63.81) .. (193.35,63.81) .. controls (193.57,63.81) and (193.75,64) .. (193.75,64.24) .. controls (193.75,64.48) and (193.57,64.67) .. (193.35,64.67) .. controls (193.13,64.67) and (192.95,64.48) .. (192.95,64.24) -- cycle ;
			%Shape: Ellipse [id:dp8217803609780121] 
			\draw  [fill={rgb, 255:red, 0; green, 0; blue, 0 }  ,fill opacity=1 ][line width=6]  (229.89,64.24) .. controls (229.89,64) and (230.07,63.81) .. (230.29,63.81) .. controls (230.52,63.81) and (230.7,64) .. (230.7,64.24) .. controls (230.7,64.48) and (230.52,64.67) .. (230.29,64.67) .. controls (230.07,64.67) and (229.89,64.48) .. (229.89,64.24) -- cycle ;
			%Shape: Ellipse [id:dp7874459453524516] 
			\draw  [fill={rgb, 255:red, 0; green, 0; blue, 0 }  ,fill opacity=1 ][line width=6]  (267.58,64.24) .. controls (267.58,64) and (267.76,63.81) .. (267.98,63.81) .. controls (268.21,63.81) and (268.39,64) .. (268.39,64.24) .. controls (268.39,64.48) and (268.21,64.67) .. (267.98,64.67) .. controls (267.76,64.67) and (267.58,64.48) .. (267.58,64.24) -- cycle ;
			%Shape: Ellipse [id:dp7182536945170612] 
			\draw  [fill={rgb, 255:red, 0; green, 0; blue, 0 }  ,fill opacity=1 ][line width=6]  (118.44,173.89) .. controls (118.44,173.66) and (118.62,173.46) .. (118.84,173.46) .. controls (119.06,173.46) and (119.24,173.66) .. (119.24,173.89) .. controls (119.24,174.13) and (119.06,174.33) .. (118.84,174.33) .. controls (118.62,174.33) and (118.44,174.13) .. (118.44,173.89) -- cycle ;
			%Shape: Ellipse [id:dp0047136538729153354] 
			\draw  [fill={rgb, 255:red, 0; green, 0; blue, 0 }  ,fill opacity=1 ][line width=6]  (155.38,173.89) .. controls (155.38,173.66) and (155.56,173.46) .. (155.78,173.46) .. controls (156.01,173.46) and (156.19,173.66) .. (156.19,173.89) .. controls (156.19,174.13) and (156.01,174.33) .. (155.78,174.33) .. controls (155.56,174.33) and (155.38,174.13) .. (155.38,173.89) -- cycle ;
			%Shape: Ellipse [id:dp9139850582308164] 
			\draw  [fill={rgb, 255:red, 0; green, 0; blue, 0 }  ,fill opacity=1 ][line width=6]  (193.44,174.3) .. controls (193.44,174.06) and (193.63,173.86) .. (193.85,173.86) .. controls (194.07,173.86) and (194.25,174.06) .. (194.25,174.3) .. controls (194.25,174.53) and (194.07,174.73) .. (193.85,174.73) .. controls (193.63,174.73) and (193.44,174.53) .. (193.44,174.3) -- cycle ;
			%Shape: Ellipse [id:dp9480289344145885] 
			\draw  [fill={rgb, 255:red, 0; green, 0; blue, 0 }  ,fill opacity=1 ][line width=6]  (230.39,174.3) .. controls (230.39,174.06) and (230.57,173.86) .. (230.79,173.86) .. controls (231.01,173.86) and (231.19,174.06) .. (231.19,174.3) .. controls (231.19,174.53) and (231.01,174.73) .. (230.79,174.73) .. controls (230.57,174.73) and (230.39,174.53) .. (230.39,174.3) -- cycle ;
			%Straight Lines [id:da6114223111722294] 
			\draw    (80.25,63.84) -- (117.94,63.84) ;
			%Straight Lines [id:da361341052732141] 
			\draw    (192.2,64.24) -- (229.89,64.24) ;
			%Straight Lines [id:da589876690059896] 
			\draw    (118.34,64.27) -- (154.88,63.84) ;
			%Straight Lines [id:da4475228736686947] 
			\draw    (230.7,64.24) -- (268.39,64.24) ;
			%Straight Lines [id:da7077883673296559] 
			\draw    (43.21,119.5) -- (80.9,119.5) ;
			%Straight Lines [id:da26660555523716023] 
			\draw    (80.9,119.5) -- (118.59,119.5) ;
			%Straight Lines [id:da3986990947534579] 
			\draw    (267.08,119.54) -- (304.77,119.54) ;
			%Straight Lines [id:da3207908416852778] 
			\draw    (230.2,119.54) -- (267.89,119.54) ;
			%Straight Lines [id:da054610758907344126] 
			\draw    (118.1,173.46) -- (155.78,173.46) ;
			%Straight Lines [id:da6790493420879957] 
			\draw    (192.7,174.3) -- (230.39,174.3) ;
			%Straight Lines [id:da48615453737994185] 
			\draw  [dash pattern={on 0.84pt off 2.51pt}]  (155.29,63.04) -- (192.95,63.44) ;
			%Straight Lines [id:da8166213771193707] 
			\draw  [dash pattern={on 0.84pt off 2.51pt}]  (118.59,118.7) -- (156.25,119.1) ;
			%Straight Lines [id:da24925085304087413] 
			\draw  [dash pattern={on 0.84pt off 2.51pt}]  (155.19,118.7) -- (192.85,119.1) ;
			%Straight Lines [id:da8014187269250024] 
			\draw  [dash pattern={on 0.84pt off 2.51pt}]  (192.14,118.7) -- (229.8,119.1) ;
			%Straight Lines [id:da9833347638271748] 
			\draw  [dash pattern={on 0.84pt off 2.51pt}]  (156.19,173.89) -- (193.85,174.3) ;
			%Straight Lines [id:da24184061692902814] 
			\draw    (267.98,63.81) -- (304.43,120.77) ;
			%Straight Lines [id:da07356387073935222] 
			\draw    (304.43,119.9) -- (230.39,174.3) ;
			%Straight Lines [id:da5660388963161784] 
			\draw    (43.21,119.5) -- (118.44,173.89) ;
			%Straight Lines [id:da4760388572663987] 
			\draw    (80.25,63.84) -- (43.21,120.37) ;
			%Shape: Ellipse [id:dp504884973862533] 
			\draw  [fill={rgb, 255:red, 0; green, 0; blue, 0 }  ,fill opacity=1 ][line width=6]  (660.16,120.16) .. controls (660.16,120.4) and (659.98,120.6) .. (659.76,120.6) .. controls (659.54,120.6) and (659.36,120.4) .. (659.36,120.16) .. controls (659.36,119.93) and (659.54,119.73) .. (659.76,119.73) .. controls (659.98,119.73) and (660.16,119.93) .. (660.16,120.16) -- cycle ;
			%Shape: Ellipse [id:dp1370482631960006] 
			\draw  [fill={rgb, 255:red, 0; green, 0; blue, 0 }  ,fill opacity=1 ][line width=6]  (623.22,120.16) .. controls (623.22,120.4) and (623.04,120.6) .. (622.82,120.6) .. controls (622.6,120.6) and (622.42,120.4) .. (622.42,120.16) .. controls (622.42,119.93) and (622.6,119.73) .. (622.82,119.73) .. controls (623.04,119.73) and (623.22,119.93) .. (623.22,120.16) -- cycle ;
			%Shape: Ellipse [id:dp03188403439048937] 
			\draw  [fill={rgb, 255:red, 0; green, 0; blue, 0 }  ,fill opacity=1 ][line width=6]  (585.53,120.16) .. controls (585.53,120.4) and (585.35,120.6) .. (585.13,120.6) .. controls (584.91,120.6) and (584.73,120.4) .. (584.73,120.16) .. controls (584.73,119.93) and (584.91,119.73) .. (585.13,119.73) .. controls (585.35,119.73) and (585.53,119.93) .. (585.53,120.16) -- cycle ;
			%Shape: Ellipse [id:dp10893955122349275] 
			\draw  [fill={rgb, 255:red, 0; green, 0; blue, 0 }  ,fill opacity=1 ][line width=6]  (548.59,120.16) .. controls (548.59,120.4) and (548.41,120.6) .. (548.19,120.6) .. controls (547.96,120.6) and (547.78,120.4) .. (547.78,120.16) .. controls (547.78,119.93) and (547.96,119.73) .. (548.19,119.73) .. controls (548.41,119.73) and (548.59,119.93) .. (548.59,120.16) -- cycle ;
			%Shape: Ellipse [id:dp6830058861974595] 
			\draw  [fill={rgb, 255:red, 0; green, 0; blue, 0 }  ,fill opacity=1 ][line width=6]  (510.53,119.76) .. controls (510.53,120) and (510.35,120.2) .. (510.12,120.2) .. controls (509.9,120.2) and (509.72,120) .. (509.72,119.76) .. controls (509.72,119.53) and (509.9,119.33) .. (510.12,119.33) .. controls (510.35,119.33) and (510.53,119.53) .. (510.53,119.76) -- cycle ;
			%Shape: Ellipse [id:dp5548857808434604] 
			\draw  [fill={rgb, 255:red, 0; green, 0; blue, 0 }  ,fill opacity=1 ][line width=6]  (473.58,120.57) .. controls (473.58,120.8) and (473.4,121) .. (473.18,121) .. controls (472.96,121) and (472.78,120.8) .. (472.78,120.57) .. controls (472.78,120.33) and (472.96,120.13) .. (473.18,120.13) .. controls (473.4,120.13) and (473.58,120.33) .. (473.58,120.57) -- cycle ;
			%Shape: Ellipse [id:dp40935996552848786] 
			\draw  [fill={rgb, 255:red, 0; green, 0; blue, 0 }  ,fill opacity=1 ][line width=6]  (435.89,120.57) .. controls (435.89,120.8) and (435.71,121) .. (435.49,121) .. controls (435.27,121) and (435.09,120.8) .. (435.09,120.57) .. controls (435.09,120.33) and (435.27,120.13) .. (435.49,120.13) .. controls (435.71,120.13) and (435.89,120.33) .. (435.89,120.57) -- cycle ;
			%Shape: Ellipse [id:dp11073720204530346] 
			\draw  [fill={rgb, 255:red, 0; green, 0; blue, 0 }  ,fill opacity=1 ][line width=6]  (398.95,119.76) .. controls (398.95,120) and (398.77,120.2) .. (398.55,120.2) .. controls (398.32,120.2) and (398.14,120) .. (398.14,119.76) .. controls (398.14,119.53) and (398.32,119.33) .. (398.55,119.33) .. controls (398.77,119.33) and (398.95,119.53) .. (398.95,119.76) -- cycle ;
			%Shape: Ellipse [id:dp06497304336652032] 
			\draw  [fill={rgb, 255:red, 0; green, 0; blue, 0 }  ,fill opacity=1 ][line width=6]  (622.72,176.26) .. controls (622.72,176.5) and (622.54,176.69) .. (622.32,176.69) .. controls (622.1,176.69) and (621.92,176.5) .. (621.92,176.26) .. controls (621.92,176.02) and (622.1,175.83) .. (622.32,175.83) .. controls (622.54,175.83) and (622.72,176.02) .. (622.72,176.26) -- cycle ;
			%Shape: Ellipse [id:dp66696782702131] 
			\draw  [fill={rgb, 255:red, 0; green, 0; blue, 0 }  ,fill opacity=1 ][line width=6]  (585.03,176.26) .. controls (585.03,176.5) and (584.85,176.69) .. (584.63,176.69) .. controls (584.41,176.69) and (584.23,176.5) .. (584.23,176.26) .. controls (584.23,176.02) and (584.41,175.83) .. (584.63,175.83) .. controls (584.85,175.83) and (585.03,176.02) .. (585.03,176.26) -- cycle ;
			%Shape: Ellipse [id:dp38031814771363104] 
			\draw  [fill={rgb, 255:red, 0; green, 0; blue, 0 }  ,fill opacity=1 ][line width=6]  (548.09,176.26) .. controls (548.09,176.5) and (547.91,176.69) .. (547.69,176.69) .. controls (547.47,176.69) and (547.29,176.5) .. (547.29,176.26) .. controls (547.29,176.02) and (547.47,175.83) .. (547.69,175.83) .. controls (547.91,175.83) and (548.09,176.02) .. (548.09,176.26) -- cycle ;
			%Shape: Ellipse [id:dp8343167446861142] 
			\draw  [fill={rgb, 255:red, 0; green, 0; blue, 0 }  ,fill opacity=1 ][line width=6]  (510.03,175.86) .. controls (510.03,176.1) and (509.85,176.29) .. (509.63,176.29) .. controls (509.4,176.29) and (509.22,176.1) .. (509.22,175.86) .. controls (509.22,175.62) and (509.4,175.43) .. (509.63,175.43) .. controls (509.85,175.43) and (510.03,175.62) .. (510.03,175.86) -- cycle ;
			%Shape: Ellipse [id:dp7123484623542555] 
			\draw  [fill={rgb, 255:red, 0; green, 0; blue, 0 }  ,fill opacity=1 ][line width=6]  (473.08,175.86) .. controls (473.08,176.1) and (472.9,176.29) .. (472.68,176.29) .. controls (472.46,176.29) and (472.28,176.1) .. (472.28,175.86) .. controls (472.28,175.62) and (472.46,175.43) .. (472.68,175.43) .. controls (472.9,175.43) and (473.08,175.62) .. (473.08,175.86) -- cycle ;
			%Shape: Ellipse [id:dp4415789929968512] 
			\draw  [fill={rgb, 255:red, 0; green, 0; blue, 0 }  ,fill opacity=1 ][line width=6]  (435.4,175.86) .. controls (435.4,176.1) and (435.22,176.29) .. (434.99,176.29) .. controls (434.77,176.29) and (434.59,176.1) .. (434.59,175.86) .. controls (434.59,175.62) and (434.77,175.43) .. (434.99,175.43) .. controls (435.22,175.43) and (435.4,175.62) .. (435.4,175.86) -- cycle ;
			%Shape: Ellipse [id:dp7591815018366901] 
			\draw  [fill={rgb, 255:red, 0; green, 0; blue, 0 }  ,fill opacity=1 ][line width=6]  (584.54,66.21) .. controls (584.54,66.45) and (584.36,66.64) .. (584.13,66.64) .. controls (583.91,66.64) and (583.73,66.45) .. (583.73,66.21) .. controls (583.73,65.97) and (583.91,65.77) .. (584.13,65.77) .. controls (584.36,65.77) and (584.54,65.97) .. (584.54,66.21) -- cycle ;
			%Shape: Ellipse [id:dp023556436101847877] 
			\draw  [fill={rgb, 255:red, 0; green, 0; blue, 0 }  ,fill opacity=1 ][line width=6]  (547.59,66.21) .. controls (547.59,66.45) and (547.41,66.64) .. (547.19,66.64) .. controls (546.97,66.64) and (546.79,66.45) .. (546.79,66.21) .. controls (546.79,65.97) and (546.97,65.77) .. (547.19,65.77) .. controls (547.41,65.77) and (547.59,65.97) .. (547.59,66.21) -- cycle ;
			%Shape: Ellipse [id:dp6990920857818113] 
			\draw  [fill={rgb, 255:red, 0; green, 0; blue, 0 }  ,fill opacity=1 ][line width=6]  (509.53,65.81) .. controls (509.53,66.04) and (509.35,66.24) .. (509.13,66.24) .. controls (508.91,66.24) and (508.73,66.04) .. (508.73,65.81) .. controls (508.73,65.57) and (508.91,65.37) .. (509.13,65.37) .. controls (509.35,65.37) and (509.53,65.57) .. (509.53,65.81) -- cycle ;
			%Shape: Ellipse [id:dp6521259630128067] 
			\draw  [fill={rgb, 255:red, 0; green, 0; blue, 0 }  ,fill opacity=1 ][line width=6]  (472.59,65.81) .. controls (472.59,66.04) and (472.41,66.24) .. (472.18,66.24) .. controls (471.96,66.24) and (471.78,66.04) .. (471.78,65.81) .. controls (471.78,65.57) and (471.96,65.37) .. (472.18,65.37) .. controls (472.41,65.37) and (472.59,65.57) .. (472.59,65.81) -- cycle ;
			%Straight Lines [id:da5625032801294048] 
			\draw    (622.72,176.26) -- (585.03,176.26) ;
			%Straight Lines [id:da4092579326619925] 
			\draw    (510.77,175.86) -- (473.08,175.86) ;
			%Straight Lines [id:da10498426954532669] 
			\draw    (584.63,175.83) -- (548.09,176.26) ;
			%Straight Lines [id:da3573203279456594] 
			\draw    (472.28,175.86) -- (434.59,175.86) ;
			%Straight Lines [id:da9064343765300158] 
			\draw    (659.76,120.6) -- (622.07,120.6) ;
			%Straight Lines [id:da5825668751896187] 
			\draw    (622.07,120.6) -- (584.38,120.6) ;
			%Straight Lines [id:da7991126411317622] 
			\draw    (435.89,120.57) -- (398.2,120.57) ;
			%Straight Lines [id:da41836049531457054] 
			\draw    (472.78,120.57) -- (435.09,120.57) ;
			%Straight Lines [id:da5736037741242561] 
			\draw    (584.88,66.64) -- (547.19,66.64) ;
			%Straight Lines [id:da741183312814103] 
			\draw    (510.28,65.81) -- (472.59,65.81) ;
			%Straight Lines [id:da8717532084604229] 
			\draw  [dash pattern={on 0.84pt off 2.51pt}]  (547.69,177.06) -- (510.03,176.66) ;
			%Straight Lines [id:da8690474913106769] 
			\draw  [dash pattern={on 0.84pt off 2.51pt}]  (584.38,121.4) -- (546.72,121) ;
			%Straight Lines [id:da3540629576619714] 
			\draw  [dash pattern={on 0.84pt off 2.51pt}]  (547.78,121.4) -- (510.12,121) ;
			%Straight Lines [id:da9329214589843771] 
			\draw  [dash pattern={on 0.84pt off 2.51pt}]  (510.84,121.4) -- (473.18,121) ;
			%Straight Lines [id:da4298834626036112] 
			\draw  [dash pattern={on 0.84pt off 2.51pt}]  (546.79,66.21) -- (509.13,65.81) ;
			%Straight Lines [id:da6884433258723379] 
			\draw    (434.99,176.29) -- (398.55,119.33) ;
			%Straight Lines [id:da9097861306924282] 
			\draw    (398.55,120.2) -- (472.59,65.81) ;
			%Straight Lines [id:da003936809415342468] 
			\draw    (659.76,120.6) -- (584.54,66.21) ;
			%Straight Lines [id:da31982931123206515] 
			\draw    (622.72,176.26) -- (659.76,119.73) ;
			
			% Text Node
			\draw (138.18,215) node [anchor=north west][inner sep=0.75pt]    {$r< p< q$};
			% Text Node
			\draw (493.52,215) node [anchor=north west][inner sep=0.75pt]    {$p< r< q$};
			% Text Node
			\draw (41.33,37.07) node [anchor=north west][inner sep=0.75pt]    {$\P_{1}: \text{ Path of length } p$};
			% Text Node
			\draw (74.67,100.73) node [anchor=north west][inner sep=0.75pt]    {$\P_{2}: \text{ Path of length } q$};
			% Text Node
			\draw (41.33,187.73) node [anchor=north west][inner sep=0.75pt]    {$\P_{3}: \text{ Path of length } r$};
			% Text Node
			\draw (394.33,37.07) node [anchor=north west][inner sep=0.75pt]    {$\P_{1}: \text{ Path of length } p$};
			% Text Node
			\draw (440.33,100.73) node [anchor=north west][inner sep=0.75pt]    {$\P_{2}: \text{ Path of length } q$};
			% Text Node
			\draw (394,187.07) node [anchor=north west][inner sep=0.75pt]    {$\P_{3}: \text{ Path of length } r$};

		\end{tikzpicture}
		
	}
	\caption{Isomorphic Bicyclic Graphs of Type-III for fixed $p,r$}
\end{figure}

Depending on the values of $p,q$ and $r$, there are a lot of different possibilities for the structure of a bicyclic type-III graph. 

For our first result, we consider a special case of the bicyclic graphs of type-III, namely, when one of the paths does not contain any vertices. The graphs in this family have an edge, where the starting and ending vertices of $\P_2$ are connected together.

Using the notation defined above, we see that $q$ must be greater than 1. Similarly, $p$, $r$ can not both be zero at the same time, since this gives us a multigraph. Using these facts, we state and prove the following.

\begin{Theorem}
	Let $\G \simeq \C_{p,q,r}$ where $r =0$ (equivalently, $p=0$), then $\beta(\G)=2$.
\end{Theorem}

\begin{proof}
	Without loss of generality we assume that $r=0$, since, considering $p=0$ gives an isomorphic graph. We consider the following two cases.
	
	\begin{proofpart}
		\mybox{When $p=1$. }
		
		Let the partitions $\V_l$ of $\V(\C_{p,q,r})$ and their representations from $\W=\left\{v_1, v_2\right\}$ be as given in the following:
		
		\resizebox{11cm}{!}{%
			\noindent\begin{minipage}{.6\linewidth}
				\begin{equation*}
					\V_l=\begin{cases}
						\left\{v_{p}\right\}, & \hspace*{0.1cm} \text{for } l=1 \vspace{0.1cm}\\ 
						\left\{ v_{p+1}\right\}, & \hspace*{0.1cm} \text{for } l=2 \vspace{0.1cm}\\
						\left\{v_{p+2}, \cdots, v_{p+\left\lceil\frac{q}{2}\right\rceil}\right\}, & \hspace*{0.1cm} \text{for } l=3 \vspace{0.1cm}\\
						\left\{v_{p+\left\lceil\frac{q}{2}\right\rceil+1}\right\}, & \hspace*{0.1cm} \text{for } l=4 \vspace{0.1cm}\\
						\left\{ v_{p+\left\lceil\frac{q}{2}\right\rceil+2,\cdots,v_{p+q}}\right\}, & \hspace*{0.1cm} \text{for } l=5 \vspace{0.1cm}.\\
						
					\end{cases}
				\end{equation*}
			\end{minipage}%
			\begin{minipage}{.6\linewidth}
				\begin{equation*}
					\RR(v_{\A} |\W)=\begin{cases}
						\left(1-{\A}, 2-{\A}\right), & \hspace*{0.1cm} \text{for } {\A}\in \V_1 \vspace{0.1cm}\\					
						\left({\A}-1, 2-{\A}\right), & \hspace*{0.1cm} \text{for } {\A}\in \V_2 \vspace{0.1cm}\\					
						\left({\A}-1, {\A}-2\right), & \hspace*{0.1cm} \text{for } {\A}\in \V_3 \vspace{0.1cm}\\					
						\left(p+q+1-{\A},{\A}-2\right), & \hspace*{0.1cm} \text{for } {\A}\in \V_4 \vspace{0.1cm}\\					
						\left(p+q+1-{\A}, p+q+1-{\A}\right), & \hspace*{0.1cm} \text{for } {\A}\in \V_5 \vspace{0.1cm}.\\
					\end{cases}
				\end{equation*}	
			\end{minipage}
		}
		
		\vspace{0.2cm}
		It can easily be verified that $\RR(v_{\A} |\W)$, is different for all vertices of $\C_{p,q,r}$. Hence, $\W$ is indeed a resolving set. 
	\end{proofpart}
	
	\begin{proofpart}
		\mybox{When $p>1$. }
		
		Let us partition $\V(\C_{p,q,r})$ into $\V_l$ along with their representations from $\W=\left\{v_1, v_{\left\lfloor\frac{p}{2}\right\rfloor+1}\right\}$ as follows:
		
		\resizebox{10.5cm}{!}{%
			\noindent\begin{minipage}{.6\linewidth}
				\begin{equation*}
					\V_l=\begin{cases}
						\left\{v_1,v_2,\cdots,v_{\left\lfloor\frac{p}{2}\right\rfloor+1}\right\}, & \hspace*{0.1cm} \text{for } l=1 \vspace{0.1cm}\\ 
						\left\{ v_{\left\lfloor\frac{p}{2}\right\rfloor+2}\right\}, & \hspace*{0.1cm} \text{for } l=2 \vspace{0.1cm}\\
						\left\{v_{\left\lfloor\frac{p}{2}\right\rfloor+3}, \cdots, v_{p}\right\}, & \hspace*{0.1cm} \text{for } l=3 \vspace{0.1cm}\\
						\left\{v_{p+1},\cdots, v_{p+\lfloor\frac{q}{2}\rfloor}\right\}, & \hspace*{0.1cm} \text{for } l=4 \vspace{0.1cm}\\
						\left\{ v_{p+\left\lceil\frac{q}{2}\right\rceil}\right\}, & \hspace*{0.1cm} \text{for } l=5 \vspace{0.1cm}\\
						\left\{v_{p+\left\lceil\frac{q}{2}\right\rceil+1},\cdots,v_{p+q}\right\}, & \hspace*{0.1cm} \text{for } l=6. \\					
					\end{cases}
				\end{equation*}
			\end{minipage}%
			\begin{minipage}{.6\linewidth}
				\begin{equation*}
					\RR(v_{\A}|\W)=\begin{cases}
						\left({\A}-1, \lfloor\frac{p}{2}\rfloor+1-{\A}\right), & \hspace*{0.1cm} \text{for } {\A}\in \V_1 \vspace{0.1cm}\\					
						\left({\A}-1, {\A}-\lfloor\frac{p}{2}\rfloor-1\right), & \hspace*{0.1cm} \text{for } {\A}\in \V_2 \vspace{0.1cm}\\					
						\left(p+3-{\A},{\A}-\lfloor\frac{p}{2}\rfloor-1\right), & \hspace*{0.1cm} \text{for } {\A}\in \V_3 \vspace{0.1cm}\\					
						\left({\A}-p,{\A}+\lfloor\frac{p}{2}\rfloor-p\right), & \hspace*{0.1cm} \text{for } {\A}\in \V_4 \vspace{0.1cm}\\
						\left({\A}-p, 2p+q-{\A}-\lfloor\frac{p}{2}\rfloor\right), & \hspace*{0.1cm} \text{for } {\A}\in \V_5 \vspace{0.1cm}\\					
						\left(p+q+2-{\A}, 2p+q-{\A}-\lfloor\frac{p}{2}\rfloor\right), & \hspace*{0.1cm} \text{for } {\A}\in \V_6 \vspace{0.1cm}.\\
					\end{cases}
				\end{equation*}	
			\end{minipage}
		}
		
		\vspace{0.2cm}
		It can be easily concluded that $\RR(v_{\XX}|\W) \neq \RR(v_{\YY}|\W)$ for $v_{\XX}, v_{\YY} \in \V_l$, and hence, $\W$ is a resolving set.
	\end{proofpart}
	These parts together with Theorem \ref{basic}-(c) provide the required proof.\qedhere
	
\end{proof}

For all other values of $p,q$ and $r$, we have the following results.

\begin{Theorem}\label{Theorem 1}
	Let $\G=\C_{p,q,r}$ where $q>p>r$ (Equivalently, $q>r>p$), then $\beta(\C_{p,q,r})=2$.
	%If $ l >k>m$  or isomorphic replacing $k=m$ then $\beta(P_{k}, P_{l}, P_{m})=2.$ where $k+l+m=n$
\end{Theorem}
\begin{proof}
	We only consider the case when $q>p>r$, since, by isomorphism, the same result is applicable to $q>r>p$. Let us consider the set $\W$ as given in the following:
	\begin{equation*}
		\W=\begin{cases}
			\left\{v_1,v_{p+2}\right\}, & \hspace*{0.1cm} \text{for } q-r=2 \vspace{0.1cm}\\ 
			\left\{v_1, v_{\left\lfloor\frac{p+r}{2}\right\rfloor+1}\right\}, & \hspace*{0.1cm} \text{for } q-r> 2 \vspace{0.1cm}.\\
		\end{cases}
	\end{equation*}
	
	We will now show that $\W$ as given above, is a metric basis for $\C_{p,q,r}(r<p<q)$. 
	%	\noindent\rule[0.5ex]{\linewidth}{1pt}
	\begin{proofpart}\label{part I}
		\mybox{When $q-r=2$.}
		
		It is obvious that, $q-p=p-r=1$. Let, $\V_l$ and their representation from $\W= \left\{v_1,v_{p+2}\right\}$ be:
		
		\resizebox{10cm}{!}{%
			\noindent\begin{minipage}{.55\linewidth}
				\begin{equation*}
					\V_l=\begin{cases}
						\left\{v_1,v_2,\cdots,v_{p-1}\right\}, & \hspace*{0.1cm} \text{for } l=1 \vspace{0.1cm}\\ 
						\left\{ v_{p}\right\}, & \hspace*{0.1cm} \text{for } l=2 \vspace{0.1cm}\\
						\left\{v_{ p+1}, v_{p+2}\right\}, & \hspace*{0.1cm} \text{for } l=3 \vspace{0.1cm}\\
						\left\{v_{p+3},\cdots,v_{p+q-1}\right\}, & \hspace*{0.1cm} \text{for } l=4 \vspace{0.1cm}\\
						\left\{v_{ p+q}\right\}, & \hspace*{0.1cm} \text{for } l=5 \vspace{0.1cm}\\
						\left\{v_{p+q+1},\cdots,v_{p+q+r}\right\}, & \hspace*{0.1cm} \text{for } l=6. \\
					\end{cases}
				\end{equation*}
			\end{minipage}%
			\begin{minipage}{.5\linewidth}
				\begin{equation*}
					\RR(v_{\A} |\R)=\begin{cases}
						\left({\A}-1, {\A}+1\right), & \hspace*{0.1cm} \text{for } {\A}\in \V_1 \vspace{0.1cm}\\
						
						\left({\A}-1, p+q-({\A}+1)\right), & \hspace*{0.1cm} \text{for } {\A}\in \V_2 \vspace{0.1cm}\\
						
						\left({\A}-p,p+2-{\A}\right), & \hspace*{0.1cm} \text{for } {\A}\in \V_3 \vspace{0.1cm}\\
						
						\left({\A}-p,{\A}-(p+2)\right), & \hspace*{0.1cm} \text{for } {\A}\in \V_4 \vspace{0.1cm}\\
						
						\left(2p+q-{\A},{\A}-(p+2)\right), & \hspace*{0.1cm} \text{for } {\A}\in \V_5 \vspace{0.1cm}\\
						
						\left({\A}+1-(p+q), {\A}+1-(p+q)\right), & \hspace*{0.1cm} \text{for } {\A}\in \V_6 \vspace{0.1cm}.\\
					\end{cases}
				\end{equation*}
			\end{minipage}
		}
		
		\vspace{0.2cm}
		Then these $\V_l$ form a partition for $\C_{p,q,r}$. It is also a straightforward task to prove that distinct vertices of $\C_{p,q,r}$ have distinct representation with respect to the set $\W=\left\{v_1,v_{p+2}\right\}$, by considering all possible cases of occurrence of $v_{\XX}$ and $v_{\YY}$, $\XX \neq \YY$, from the partitions $\V_l$ defined above. For the sake of completeness, we provide the proof here but will omit it later on in other theorems.
		
		It is simple to prove that when $v_{\XX},v_{\YY}, \XX \neq \YY$ are from the same partition $\V_l$, then
		\begin{equation*}
			\RR(v_{\XX}|\W)\neq \RR(v_{\YY}|\W).
		\end{equation*}
		
		Again, for completeness, we solve the case when $v_{\XX} \neq v_{\YY}$ and both are from $\V_6$. Assuming on the contrary that $\RR(v_{\XX}|\W)= \RR(v_{\YY}|\W)$, we get 
		\begin{equation*}
			\left({\XX}+1-(p+q), {\XX}+1-(p+q)\right)=\left({\YY}+1-(p+q), {\YY}+1-(p+q)\right).
		\end{equation*}
		
		This in turn gives us $\XX=\YY$, a contradiction.
		
		For the remaining cases, we proceed as follows.
		\begin{mycases}
			\begin{subcases}

				\case {When $v_{\XX} \in \V_1 $, $v_{\YY} \in \V_2 $.}
				We assume that $\RR(v_{\XX}|\W)\neq \RR(v_{\YY}|\W)$. If this is not the case then
				$\RR(v_{\XX}|\W)=\RR(v_{\YY}|\W)$ gives us $({\XX}-1, {\XX}+1)=({\YY}-1, p+q-({\YY}+1))$.
				Equating the corresponding terms and solving, we get ${\XX}={\YY}$ and ${\XX}+{\YY}=p+q-2$.
				Solving these for ${\XX}$ gives, ${\XX}=\frac{p+q}{2}-1$, a contradiction to the fact that ${\XX}$ is an integer. Hence $v_{\XX}, v_{\YY}$ have distinct representations.
				
				\case {When $v_{\XX} \in \V_1$ and $v_{\YY} \in \V_3 $.}
				If we consider $\RR(v_{\XX}|\W) = \RR(v_{\YY}|\W)$, we get ${\YY}-{\XX}=p-1$ and ${\YY}+{\XX}=p+1$.
				Solving for ${\YY}$, we get, ${\YY}=p$, but $v_{\YY} \in \V_3=\{v_{p+1},v_{p+2}\}$. This provides us with a contradiction.

				\case {When $v_{\XX} \in \V_1$, $v_{\YY} \in \V_4 $.}
				We again assume that $\RR(v_{\XX}|\W)\neq \RR(v_{\YY}|\W)$, if not, then $\RR(v_{\XX}|\W)=\RR(v_{\YY}|\W)$
				$\Rightarrow$ $\left ({\XX}-1, {\XX}+1\right) = \left ({\YY}-p,{\YY}-(p+2\right))$
				$\Rightarrow$ ${\XX}-1={\YY}-p$ and  ${\XX}+1={\YY}-(p+2)$
				$\Rightarrow$ ${\YY}-{\XX}=p-1$ and ${\YY}-{\XX}=p+3$, a contradiction.
				
				\case {When $v_{\XX} \in \V_1$, $v_{\YY} \in \V_5 $.}
				If we take $\RR(v_{\XX}|\W)=\RR(v_{\YY}|\W)$, we get ${\XX}-1=2p+q-{\YY}$ and  ${\XX}+1={\YY}-(p+2)$. Rearranging the terms gives us, ${\XX}+{\YY}=2p+q+1$ and ${\XX}-{\YY}=p+3$. Solving these two for ${\XX}$, we get ${\XX}=\frac{p+q}{2}-1 > \frac{p+p}{2}-1=p-1$, since $q >p$. This gives us the contradiction ${\XX} >p-1$, since $v_{\XX} \in \V_1=\{v_1,v_2, \cdots, v_{p-1}\}$.
				
				\case {When $v_{\XX} \in \V_1$, $v_{\YY} \in \V_6 $.}
				Assuming $\RR(v_{\XX}|\W)=\RR(v_{\YY}|\W)$ and equating the corresponding distances, we get, ${\YY}-{\XX}=p+q-2$ and ${\YY}-{\XX}=p+q$, a contradiction.
				
				\case {When $v_{\XX} \in \V_2$, $v_{\YY} \in \V_3 $.}\label{q+p-4=0}
				We again assume that $\RR(v_{\XX}|\W)\neq \RR(v_{\YY}|\W)$, for if not, then $\RR(v_{\XX}|\W)=\RR(v_{\YY}|\W) \implies {\XX}-1={\YY}-p$ and  $p+q-({\XX}+1)=p+2-{\YY}$.
				Rearranging the terms, we get, ${\YY}-{\XX}=p-1$ and ${\YY}-{\XX}=3-q$. Subtracting these two gives us $q+p-4=0 \implies q+p=4$. On the other hand, we know that $q-p=1$. Solving these two gives $q=\frac{5}{2}$, a contradiction. 
				
				\case {When $v_{\XX} \in \V_2$, $v_{\YY} \in \V_4 $.}
				Considering $\RR(v_{\XX}|\W)=\RR(v_{\YY}|\W)$ and equating and solving the corresponding terms, we obtain ${\YY}-{\XX}=p-1$ and ${\YY}+{\XX}=2p+q+1$.
				The fact that $v_{\XX} \in \V_2$ and $v_{\YY} \in \V_4 $ contradicts with the solution ${\YY}+{\XX}=2p+q+1$. Hence $\RR(v_{\XX}|\W)\neq \RR(v_{\YY}|\W)$.
				
				\case {When $v_{\XX} \in \V_2$, $v_{\YY}\in \V_5 $.}
				If we again consider $\RR(v_{\XX}|\W)=\RR(v_{\YY}|\W)$, we get ${\YY}+{\XX}=2p+q+1$ which is a contradiction to the fact that $v_{\XX} \in \V_2$ and $v_{\YY}\in \V_5 $. Consequently, $\RR(v_{\XX}|\W)\neq \RR(v_{\YY}|\W)$.
				
				\case {When $v_{\XX} \in \V_2$, $v_{\YY} \in \V_6 $.}	
				We claim that $\RR(v_{\XX}|\W)\neq \RR(v_{\YY}|\W)$, for if not, then $\RR(v_{\XX}|\W)=\RR(v_{\YY}|\W)$ gives, ${\YY}-{\XX}=p+q-2$ and ${\YY}+{\XX}=2p+2q-2$.
				Solving for ${\XX}$ gives, ${\XX}=\frac{p+q}{2}$. Now $q-p=1 \implies q=p+1$. Putting this value in ${\XX}=\frac{p+q}{2}$, we get ${\XX}=p+\frac{1}{2}$, a contradiction.
				
				\case {When $v_{\XX} \in \V_3$, $v_{\YY} \in \V_4 $.}
				Considering $\RR(v_{\XX}|\W)=\RR(v_{\YY}|\W)$ and proceeding as before to equate and solve corresponding terms, we reach the conclusion ${\YY}={\XX}$ and ${\YY}+{\XX}=2p+4 \implies {\YY}=p+2$, a contradiction to the fact that $v_{\YY} \in \V_4=\left\{v_{p+3},\cdots,v_{p+q-1}\right\}$.
				
				\case {When $v_{\XX} \in \V_3$, $v_{\YY} \in \V_5 $.}
				We again claim that $\RR(v_{\XX}|\W)\neq \RR(v_{\YY}|\W)$, for if not, then $\RR(v_{\XX}|\W)=\RR(v_{\YY}|\W)$ gives ${\YY}+{\XX}=3p+q$ and ${\YY}+{\XX}=2p+4$. Subtracting these equations gives $q+p-4=0$ which is a contradiction, by the explanation already provided in the Case \ref{q+p-4=0} above.
				
				\case {When $v_{\XX} \in \V_3$, $v_{\YY} \in \V_6 $.}
				Assuming $\RR(v_{\XX}|\W)=\RR(v_{\YY}|\W)$ for this case and solving as before, we get ${\YY}=p+q$, a contradiction to the fact that $v_{\YY} \in \V_6$.
				
				\case {When $v_{\XX} \in \V_4$, $v_{\YY} \in \V_5 $.}
				Considering $\RR(v_{\XX}|\W)=\RR(v_{\YY}|\W)$ produces ${\YY}+{\XX}=3p+q$ and ${\YY}={\XX} \implies {\YY}=\frac{3p+q}{2}$. Since $q=p+1$, we get ${\YY}=2p+\frac{1}{2}$, a contradiction.
				
				\case {$v_{\XX} \in \V_4$, $v_{\YY} \in \V_6 $.}
				We assume that $\RR(v_{\XX}|\W)\neq \RR(v_{\YY}|\W)$, for if not, then $\RR(v_{\XX}|\W)=\RR(v_{\YY}|\W)$ gives the contradiction ${\YY}-{\XX}=q-1$ and ${\YY}-{\XX}=q-3$. 
				
				\case {$v_{\XX} \in \V_5$, $v_{\YY} \in \V_6 $.}
				If we take $\RR(v_{\XX}|\W)=\RR(v_{\YY}|\W)$, we get ${\YY}+{\XX}=3p+2q-1$ and ${\YY}-{\XX}=q-3$.  Solving for ${\YY}$ and using the fact that $q=p+1$, we arrive at the contradiction that ${\YY}$ is a fraction. Hence $\RR(v_{\XX}|\W)\neq \RR(v_{\YY}|\W)$.
			\end{subcases}
		\end{mycases}
	\end{proofpart}
	%\noindent\rule[0.5ex]{\linewidth}{1pt}
	\begin{proofpart}\label{part II}		
		\mybox{When $q-r>2$ and $p-r=1$. }
		
		Let the partitions $\V_l$ of $\V(\C_{p,q,r})$ and their representations from $\W=\left\{v_1, v_{\left\lfloor\frac{p+r}{2}\right\rfloor+1}\right\}$ be as:
		
		\resizebox{10cm}{!}{%
			\noindent\begin{minipage}{.65\linewidth}
				\begin{equation*}
					\V_l=\begin{cases}
						\left\{v_1,v_2,\cdots,v_{p}\right\}, & \hspace*{0.1cm} \text{for } l=1 \vspace{0.1cm}\\ 
						\left\{ v_{p+1},\cdots, v_{p+\left\lfloor\frac{q-r}{2}\right\rfloor}\right\}, & \hspace*{0.1cm} \text{for } l=2 \vspace{0.1cm}\\
						\left\{v_{p+1+\left\lfloor\frac{q-r}{2}\right\rfloor}, \cdots, v_{p+q-\left\lfloor\frac{q-r}{2}\right\rfloor}\right\}, & \hspace*{0.1cm} \text{for } l=3 \vspace{0.1cm}\\
						\left\{v_{p+q+1-\left\lfloor\frac{q-r}{2}\right\rfloor},\cdots, v_{p+q}\right\}, & \hspace*{0.1cm} \text{for } l=4 \vspace{0.1cm}\\
						\left\{v_{p+q+1},\cdots,v_{p+q+r}\right\}, & \hspace*{0.1cm} \text{for } l=5. \\					
					\end{cases}
				\end{equation*}
			\end{minipage}%
			\begin{minipage}{.5\linewidth}
				\begin{equation*}
					\RR(v_{\A} |\R)=\begin{cases}
						\left({\A}-1, p-{\A}\right), & \hspace*{0.1cm} \text{for } {\A}\in \V_1 \vspace{0.1cm}\\					
						\left({\A}-p, {\A}-1\right), & \hspace*{0.1cm} \text{for } {\A}\in \V_2 \vspace{0.1cm}\\					
						\left({\A}-p,p+q+1-{\A}\right), & \hspace*{0.1cm} \text{for } {\A}\in \V_3 \vspace{0.1cm}\\					
						\left(2p+q-{\A},p+q+1-{\A}\right), & \hspace*{0.1cm} \text{for } {\A}\in \V_4 \vspace{0.1cm}\\					
						\left({\A}+1-(p+q), p+q+r+2-{\A}\right), & \hspace*{0.1cm} \text{for } {\A}\in \V_5 \vspace{0.1cm}.\\
					\end{cases}
				\end{equation*}	
			\end{minipage}
		}
		
		\vspace{0.2cm}
		It can easily be verified that the representation of all the vertices, with respect to the basis set $\W$, is different. Hence, $\W$ is a resolving set.
		
	\end{proofpart}
	%\noindent\rule[0.5ex]{\linewidth}{1pt}
	
	\begin{proofpart}\label{part III}
		\mybox{When $q-r >2 $ and $p-r > 1$.}
		
		Since $p-r > 1$ and $p >r $, w get that $p-r > 2$. Let us partition the vertices of $\C_{p,q,r}$ into $\V_l$ as given in the following:

		%\resizebox{12cm}{!}{%
			\noindent\begin{minipage}{.6\linewidth}
				\begin{equation*}
					\V_l=\begin{cases}
						\left\{v_1,v_2,\cdots, v_{\left\lfloor\frac{p+r}{2}\right\rfloor+1}\right\}, & \hspace*{0.1cm} \text{for } l=1 \vspace{0.1cm}\\ 
						\left\{ v_{\left\lfloor\frac{p+r}{2}\right\rfloor+2},\cdots, v_{p+1-\left\lfloor\frac{p-r}{2}\right\rfloor}\right\}, & \hspace*{0.1cm} \text{for } l=2 \vspace{0.1cm}\\
						\left\{v_{p+2-\left\lfloor\frac{p-r}{2}\right\rfloor}, \cdots, v_{p}\right\}, & \hspace*{0.1cm} \text{for } l=3 \vspace{0.1cm}\\
						\left\{v_{p+1},\cdots, v_{p+\left\lfloor\frac{q-r}{2}\right\rfloor}\right\}, & \hspace*{0.1cm} \text{for } l=4 \vspace{0.1cm}\\
						\left\{v_{p+1+\left\lfloor\frac{q-r}{2}\right\rfloor},\cdots, v_{p+q-\left\lfloor\frac{q-r}{2}\right\rfloor}\right\}, & \hspace*{0.1cm} \text{for } l=5 \\
						\left\{v_{p+q+1-\left\lfloor\frac{q-r}{2}\right\rfloor},\cdots, v_{p+q}\right\}, & \hspace*{0.1cm} \text{for } l=6 \\
						\left\{v_{p+q+1},\cdots,v_{p+q+r}\right\}, & \hspace*{0.1cm} \text{for } l=7. \\
						
					\end{cases}
				\end{equation*}
			\end{minipage}%
		
		Since $\W=\left\{v_1, v_{\left\lfloor\frac{p+r}{2}\right\rfloor+1}\right\}$, considering the representation of all partitions $\V_l$ from $\W$, we get the following:
		
			\begin{minipage}{.5\linewidth}
				\begin{equation*}
					\RR(v_{\A}| \R)=\begin{cases}
						\left({\A}-1, \left\lfloor\frac{p+r}{2}\right\rfloor+1-{\A}\right), & \hspace*{0.1cm} \text{for } {\A}\in \V_1 \vspace{0.1cm}\\
						
						\left({\A}-1,{\A}-\left\lfloor\frac{p+r}{2}\right\rfloor-1\right), & \hspace*{0.1cm} \text{for } {\A}\in \V_2 \vspace{0.1cm}\\
						
						\left(p+r+3-{\A},{\A}-\left\lfloor\frac{p+r}{2}\right\rfloor-1\right), & \hspace*{0.1cm} \text{for } {\A}\in \V_3 \vspace{0.1cm}\\
						
						\left({\A}-p,{\A}+\left\lfloor\frac{p+r}{2}\right\rfloor-p\right), & \hspace*{0.1cm} \text{for } {\A}\in \V_4 \vspace{0.1cm}\\
						\left({\A}-p,2p+q-\left\lfloor\frac{p+r}{2}\right\rfloor-{\A}\right), & \hspace*{0.1cm} \text{for } {\A}\in \V_5 \vspace{0.1cm}\\
						\left(p+q+r+2-{\A},2p+q-\left\lfloor\frac{p+r}{2}\right\rfloor-{\A}\right), & \hspace*{0.1cm} \text{for } {\A}\in \V_6 \vspace{0.1cm}\\
						
						\left({\A}+1-(p+q), 2p+q+r+1-\left\lfloor\frac{p+r}{2}\right\rfloor-{\A}\right), & \hspace*{0.1cm} \text{for } {\A}\in \V_7 \vspace{0.1cm}.\\
					\end{cases}
				\end{equation*}
			\end{minipage}
	%	}
		
		\vspace{0.2cm}
		Using the above given partitions and their representations, it can be easily verified that $\W$ is indeed a resolving set in this case.
		
		The above three cases together with Theorem \ref{basic}, (c) ensure that $\beta(\C_{p,q,r})=2$.\qedhere

	\end{proofpart}
\end{proof}

\begin{Theorem}\label{Theorem 2}
	Let $\G=\C_{p,q,r}$ where $p=q$ and $q \lesseqqgtr r  $ (Equivalently, if $r=q$ and $q \lesseqqgtr  p$), then 
	\[
	\beta(\G)=\begin{cases}
		2 & q-r \neq 2 \\
		3 & q-r = 2. \\
	\end{cases}
	\]
\end{Theorem}
\begin{proof}
	Let us suppose that $q \lesseqqgtr r$. Let us consider $\W \subset \V(\G)$ as follows.
	\begin{equation*}
		\W=\begin{cases}
			\{v_1,v_{p}\}, & \hspace*{0.1cm} \text{for } q-r<2 \vspace{0.1cm}\\ 
			\{v_1, v_2,v_{p+2}\}, & \hspace*{0.1cm} \text{for } q-r=2 \vspace{0.1cm}\\	
			\{v_1, v_{\left\lfloor\frac{p+r}{2}\right\rfloor+1}\}, & \hspace*{0.1cm} \text{for } q-r>2 \vspace{0.1cm}.\\
			
		\end{cases}
	\end{equation*}
	
	The three parts are then discussed separately.
	
	%\noindent\rule[0.5ex]{\linewidth}{1pt}
	
	\begin{proofpart}
		\mybox{When $q-r < 2$.}\label{part IV}
		
		Let $\W=\{v_1,v_p\}$, and let $\V_l$ and their representation from $\W$ be as given in the following:
		
		\resizebox{10.5cm}{!}{%
			\noindent\begin{minipage}{.75\linewidth}
				\begin{equation*}
					\V_l=\begin{cases}
						\left\{v_1,v_2,\cdots,v_{p}\right\}, & \hspace*{0.2cm} \text{for } l=1 \vspace{0.1cm}\\ 
						
						\left\{v_{p+1},\cdots,v_{p+q}\right\}, & \hspace*{0.2cm} \text{for } l=2 \vspace{0.1cm}\\
						
						\left\{v_{p+q+1},\cdots,v_{p+q+\left\lceil\frac{r-q}{2}\right\rceil}\right\}, & \hspace*{0.2cm} \text{for } l=3 \\
						\left\{v_{p+q+1+\left\lceil\frac{r-q}{2}\right\rceil},\cdots,v_{p+q+r-\left\lceil\frac{r-q}{2}\right\rceil}\right\}, & \hspace*{0.2cm} \text{for } l=4 \\
						\left\{v_{p+q+r+1-\left\lceil\frac{r-q}{2}\right\rceil},\cdots,v_{p+q+r}\right\}, & \hspace*{0.2cm} \text{for } l=5. \\
					\end{cases}
				\end{equation*}
			\end{minipage}%
			\begin{minipage}{.6\linewidth}
				\begin{equation*}
					\RR(v_{\A} | \R)=\begin{cases}
						\left({\A}-1, p-{\A}\right), & \hspace*{0.2cm} \text{for } {\A}\in \V_1 \vspace{0.1cm}\\
						
						\left({\A}-p, p+q+1-{\A}\right), & \hspace*{0.2cm} \text{for } {\A}\in \V_2 \vspace{0.1cm}\\
						
						\left({\A}+1-(p+q),{\A}-q\right), & \hspace*{0.2cm} \text{for } {\A}\in \V_3 \vspace{0.1cm}\\
						\left({\A}+1-(p+q),p+q+r+2-{\A}\right), & \hspace*{0.2cm} \text{for } {\A}\in \V_4 \vspace{0.1cm}\\
						\left(2p+q+r+1-{\A},p+q+r+2-{\A}\right), & \hspace*{0.2cm} \text{for } {\A}\in \V_5 \vspace{0.1cm}.\\
					\end{cases}
				\end{equation*}
			\end{minipage}
		}
		
		\vspace{0.2cm}
		It can easily be verified that $\W$ is a resolving set here. This, together with Theorem \ref{basic}(c), gives us the result for this case.

	\end{proofpart}
	
	%\noindent\rule[0.5ex]{\linewidth}{1pt}
	
	\begin{proofpart}
		
		\mybox{When $q-r=2$.}\label{part V}
		
		Let $\W=\{v_1,v_2,v_{p+2}\}$ and let the partitions $\V_l$ of $\V(\G)$ and their representations from $\W$ be as provided below:
		
		\resizebox{9.5cm}{!}{%
			\noindent\begin{minipage}{.61\linewidth}
				\begin{equation*}
					\V_l=\begin{cases}
						\left\{v_1\right\}, & \hspace*{0.2cm} \text{for } l=1 \vspace{0.1cm}\\ 
						\left\{v_2,v_3,\cdots,v_{p-1}\right\}, & \hspace*{0.2cm} \text{for } l=2 \vspace{0.1cm}\\
						\left\{v_{p}\right\}, & \hspace*{0.2cm} \text{for } l=3 \vspace{0.1cm}\\  
						
						\left\{v_{p+1}\right\}, & \hspace*{0.2cm} \text{for } l=4 \vspace{0.1cm}\\ 
						\left\{v_{p+2},v_{p+3},\cdots,v_{p+q-1}\right\}, & \hspace*{0.2cm} \text{for } l=5 \vspace{0.1cm}\\ 
						\left\{v_{p+q}\right\}, & \hspace*{0.2cm} \text{for } l=6 \vspace{0.1cm}\\
						
						\left\{v_{p+q+1},\cdots,v_{p+q+r}\right\}, & \hspace*{0.2cm} \text{for } l=7. \\
					\end{cases}
				\end{equation*}
			\end{minipage}%
			\begin{minipage}{.6\linewidth}
				\begin{equation*}
					\RR(v_{\A} | \R)=\begin{cases}
						\left({\A}-1, 2-{\A}, {\A}+1\right), & \hspace*{0.2cm} \text{for } {\A}\in \V_1 \vspace{0.1cm}\\
						
						\left({\A}-1, {\A}-2, {\A}+1\right), & \hspace*{0.2cm} \text{for } {\A}\in \V_2 \vspace{0.1cm}\\
						\left({\A}-1, {\A}-2, {\A}-1\right), & \hspace*{0.2cm} \text{for } {\A}\in \V_3 \vspace{0.1cm}\\
						
						\left({\A}-p,{\A}+1-p,p+2-{\A}\right), & \hspace*{0.2cm} \text{for } {\A}\in \V_4 \vspace{0.1cm}\\
						\left({\A}-p,{\A}+1-p,{\A}-(p+2)\right), & \hspace*{0.2cm} \text{for } {\A}\in \V_5 \vspace{0.1cm}\\
						\left({\A}-p,{\A}-(1+p),{\A}-(p+2)\right), & \hspace*{0.2cm} \text{for } {\A}\in \V_6 \vspace{0.1cm}\\
						\left({\A}+1-(p+q),{\A}+2-(p+q),{\A}+1-(p+q)\right), & \hspace*{0.2cm} \text{for } {\A}\in \V_7 \vspace{0.1cm}.\\
					\end{cases}
				\end{equation*}
			\end{minipage}
		}
		
		\vspace{0.2cm}
		
		Using the above stated information, we can easily conclude that no two vertices of $\V(\G)$ have same representation with respect to $\W$. This ensures that $\W$ is a resolving set for this case. To see that $\W$ is a minimal resolving set, let us consider the three cases of $\W-\{w \in \W\}$.
		
		\begin{mycases}
			\case When we remove $v_1$ from $\W$.
			
			Let $\W'=\W-v_1=\{v_2,v_{p+2}\}$. Let us take the vertices $v_{p+q}$ and $v_{p+q+r-1}$ from figure \ref{typeiiibicyclic}. Using the representations given above and the fact that $q-r=2$, it can be easily shown that $\RR(v_{p+q}|\W)=\RR(v_{p+q+r-1}|\W)$.
			
			\case When we remove $v_2$ from $\W$.
			
			Let $\W'=\W-v_2=\{v_1,v_{p+2}\}$. Let us consider the vertices $v_{p}$ and $v_{p+q+r}$. Using the representations given above and the fact that $q-r=p-r=2$, it can be easily shown that $\RR(v_{p}|\W)=\RR(v_{p+q+r}|\W)$.
			
			\case When we remove $v_{p+2}$ from $\W$.
			
			Let $\W'=\W-v_{p+2}=\{v_1,v_2\}$. Considering the vertices $v_{p+2}$ and $v_{p+q+1}$, one can see that they have same representation with respect to $\W'$.
		\end{mycases}
		Since $\W-\{w \in \W\}$ is not a resolving set, we conclude that $\W$ is a minimal resolving set and hence, $\beta(\G)=3$ for $\G=\C_{p,q,r}$, where $q-r=2$.

	\end{proofpart}
	%\noindent\rule[0.5ex]{\linewidth}{1pt}
	
	\begin{proofpart}

		\mybox{When $q-r>2$.}\label{part VI}
		
		Let $\W=\left\{v_1, v_{\left\lfloor\frac{p+r}{2}\right\rfloor+1}\right\}$ and let the partitions of $\V(\G)$  with their representations from $\W$ be as follows:
		
		\resizebox{9.5cm}{!}{%
			\noindent\begin{minipage}{.65\linewidth}
				\begin{equation*}
					\V_l=\begin{cases}
						\left\{v_1,\cdots,v_{\left\lfloor\frac{p+r}{2}\right\rfloor+1}\right\}, & \hspace*{0.2cm} \text{for } l=1 \vspace{0.1cm}\\ 
						\left\{v_{\left\lfloor\frac{p+r}{2}\right\rfloor+2},\cdots, v_{p+1-\left\lfloor\frac{p-r}{2}\right\rfloor}\right\}, & \hspace*{0.2cm} \text{for } l=2 \vspace{0.1cm}\\
						\left\{v_{p+2-\left\lfloor\frac{p-r}{2}\right\rfloor},\cdots, v_{p}\right\}, & \hspace*{0.2cm} \text{for } l=3 \vspace{0.1cm}\\  
						
						\left\{v_{p+1},\cdots, v_{p+\left\lfloor\frac{p-r}{2}\right\rfloor}\right\}, & \hspace*{0.2cm} \text{for } l=4 \vspace{0.1cm}\\ 
						\left\{v_{p+1+\left\lfloor\frac{q-r}{2}\right\rfloor},\cdots,v_{p+q-\left\lfloor\frac{q-r}{2}\right\rfloor}\right\}, & \hspace*{0.2cm} \text{for } l=5 \vspace{0.1cm}\\ 
						\left\{v_{p+q+1-\left\lfloor\frac{q-r}{2}\right\rfloor},\cdots,v_{p+q}\right\}, & \hspace*{0.2cm} \text{for } l=6 \vspace{0.1cm}\\
						
						\left\{v_{p+q+1},\cdots,v_{p+q+r}\right\}, & \hspace*{0.2cm} \text{for } l=7. \\
					\end{cases}
				\end{equation*}
			\end{minipage}%
			\begin{minipage}{.5\linewidth}
				\begin{equation*}
					\RR(v_{\A} | \R)=\begin{cases}
						\left({\A}-1, 1+\left\lfloor\frac{p+r}{2}\right\rfloor-{\A}\right), & \hspace*{0.2cm} \text{for } {\A}\in \V_1 \vspace{0.1cm}\\
						
						\left ({\A}-1,{\A}-\left( 1+\left\lfloor\frac{p+r}{2}\right\rfloor\right)\right), & \hspace*{0.2cm} \text{for } {\A}\in \V_2 \vspace{0.1cm}\\
						\left(p+r+3-{\A},{\A}-\left( 1+\left\lfloor\frac{p+r}{2}\right\rfloor\right)\right), & \hspace*{0.2cm} \text{for } {\A}\in \V_3 \vspace{0.1cm}\\
						
						\left({\A}-p,{\A}+\left\lfloor\frac{p+r}{2}\right\rfloor-p\right), & \hspace*{0.2cm} \text{for } {\A}\in \V_4 \vspace{0.1cm}\\
						\left({\A}-p,3p-\left(\left\lfloor\frac{p+r}{2}\right\rfloor+{\A}\right)\right), & \hspace*{0.2cm} \text{for } {\A}\in \V_5 \vspace{0.1cm}\\
						\left(p+q+r+2-{\A},3p-\left( \left\lfloor\frac{p+r}{2}\right\rfloor+{\A}\right)\right), & \hspace*{0.2cm} \text{for } {\A}\in \V_6 \vspace{0.1cm}\\
						\left ({\A}+1-(p+q),2p+q+r+1-\left(\left\lfloor\frac{p+r}{2}\right\rfloor+j\right)\right)\, & \hspace*{0.2cm} \text{for } {\A}\in \V_7 \vspace{0.1cm}.\\
					\end{cases}
				\end{equation*}
			\end{minipage}
		}
		
		\vspace{0.2cm}
		
		Again, we can easily verify that $\W=\left\{v_1,v_{\left\lfloor\frac{p+r}{2}\right\rfloor+1}\right\}$ is a resolving set here. Together with Theorem \ref{basic}-(c), we obtain our result for this case.\qedhere

	\end{proofpart}
\end{proof}

\begin{Theorem}\label{Theorem 3}
	Let $\G=\C_{p,q,r}$, where $p>q$, $p>r$ and $q \lesseqqgtr r$ (Equivalently, $r > q$, $r>p,$ and $q \lesseqqgtr p$), then $\beta(\G)=2.$
\end{Theorem}

%\begin{Theorem}
%	If $ k>l,k>m$ and $l\lesseqgtr  m$  or issomorphic replacing $k=m$ then mertic dimention of graph 
%	$	\beta(\G)=2.$ where $k+l+m=n$
%\end{Theorem}
\begin{proof}
	Without loss of generality, we assume that $\max\{q,r\}=q$. Let us consider the set $\W$ as given in the following:
	\begin{equation*}
		\W=\begin{cases}
			\left\{v_1,v_{\left\lfloor\frac{p+q}{2}\right\rfloor}\right\}, & \hspace*{0.1cm} \text{for } q-r<2 \vspace{0.1cm}\\ 
			\left\{v_1,v_{p+2}\right\}, & \hspace*{0.1cm} \text{for } q-r=2 \vspace{0.1cm}\\ 
			\left\{v_1, v_{\left\lfloor\frac{p+r}{2}\right\rfloor+1}\right\}, & \hspace*{0.1cm} \text{for } q-r>2 \vspace{0.1cm}.\\
		\end{cases}
	\end{equation*}
	We now consider the representation of vertices of $\C_{p,q,r}$ from these $\W$ for all three cases.
	%	\noindent\rule[0.5ex]{\linewidth}{1pt}
	\begin{proofpart}
		\mybox{When $q-r<2$.}\label{part VII}
		
		Let us partition vertices of $\C_{p,q,r}$ into $\V_l$ as given by the following. We also provide the representation of $\V_l$ from $\W=\left\{v_1,v_{\left\lfloor\frac{p+q}{2}\right\rfloor}\right\}$.
		
		\resizebox{9.5cm}{!}{%
			\noindent\begin{minipage}{.7\linewidth}
				\begin{equation*}
					\V_l=\begin{cases}
						\left\{v_1,\cdots,v_{\left\lfloor\frac{p+q}{2}\right\rfloor}\right\}, & \hspace*{0.1cm} \text{for } l=1 \vspace{0.1cm}\\ 
						\left\{v_{\left\lfloor\frac{p+q}{2}\right\rfloor+1},\cdots, v_{p-\left\lceil\frac{p-q}{2}\right\rceil}\right\}, & \hspace*{0.1cm} \text{for } l=2 \vspace{0.1cm}\\
						\left\{v_{p+1-\left\lceil\frac{p-q}{2}\right\rceil},\cdots, v_{p}\right\}, & \hspace*{0.1cm} \text{for } l=3 \vspace{0.1cm}\\  
						
						\left\{v_{p+1}\right\}, & \hspace*{0.1cm} \text{for } l=4 \vspace{0.1cm}\\ 
						\left\{v_{p+2},\cdots,v_{p+q}\right\}, & \hspace*{0.1cm} \text{for } l=5 \vspace{0.1cm}\\ 
						\left\{v_{p+q+1},\cdots v_{p+q+1+\left\lfloor\frac{r-q}{2}\right\rfloor}\right\}, & \hspace*{0.1cm} \text{for } l=6 \vspace{0.1cm}\\
						
						\left\{v_{p+q+2+\left\lfloor\frac{r-q}{2}\right\rfloor},\cdots,v_{p+q+r-\left\lceil\frac{r-q}{2}\right\rceil}\right\}, & \hspace*{0.1cm} \text{for } l=7 \\
						\left\{v_{p+q+r+1-\left\lceil\frac{r-q}{2}\right\rceil},\cdots,v_{p+q+r}\right\}, & \hspace*{0.1cm} \text{for } l=8. \\
					\end{cases}
				\end{equation*}
			\end{minipage}%
			\begin{minipage}{.5\linewidth}
				\begin{equation*}
					\RR(v_{\A} | \R)=\begin{cases}
						\left({\A}-1, \left\lfloor\frac{p+q}{2}\right\rfloor-{\A}\right), & \hspace*{0.1cm} \text{for } {\A}\in \V_1 \vspace{0.1cm}\\
						\left({\A}-1, {\A}-\left\lfloor\frac{p+q}{2}\right\rfloor\right), & \hspace*{0.1cm} \text{for } {\A}\in \V_2 \vspace{0.1cm}\\
						
						\left(p+q+1-{\A}, {\A}-\left\lfloor\frac{p+q}{2}\right\rfloor\right), & \hspace*{0.1cm} \text{for } {\A}\in \V_3 \vspace{0.1cm}\\
						
						\left({\A}-p,\left\lfloor\frac{p+q}{2}\right\rfloor\right), & \hspace*{0.1cm} \text{for } {\A}\in \V_4 \vspace{0.1cm}\\
						
						\left({\A}-p,2p+q-\left\lfloor\frac{p+q}{2}\right\rfloor-{\A}\right), & \hspace*{0.1cm} \text{for } {\A}\in \V_5 \vspace{0.1cm}\\
						
						\left({\A}+1-(p+q), {\A}+\left\lfloor\frac{p+q}{2}\right\rfloor-(p+q)\right), & \hspace*{0.1cm} \text{for } {\A}\in \V_6 \vspace{0.1cm}\\
						\left({\A}+1-(p+q),2p+q+r+2-({\A}+\left\lfloor\frac{p+q}{2}\right\rfloor)\right), & \hspace*{0.1cm} \text{for } {\A}\in \V_7 \vspace{0.1cm}\\
						\left(p+2q+r+1-{\A},2p+q+r+2-({\A}+\left\lfloor\frac{p+q}{2}\right\rfloor)\right), & \hspace*{0.1cm} \text{for } {\A}\in \V_8 \vspace{0.1cm}\\
					\end{cases}
				\end{equation*}
			\end{minipage}
		}
		
		\vspace{0.2cm}

	\end{proofpart}
	
	%\noindent\rule[0.5ex]{\linewidth}{1pt}
	
	\begin{proofpart}
		\mybox{When $q-r=2$.}\label{part VIII}
		
		Considering $\V_l$ and their representations from $\W=\left\{v_1,v_{p+2}\right\}$, as given in the following, it can easily be verified that $\W$, as provided here, is a resolving set in this case.
		
		\resizebox{10.5cm}{!}{%
			\noindent\begin{minipage}{.65\linewidth}
				\begin{equation*}
					\V_l=\begin{cases}
						\left\{v_1,\cdots,v_{\left\lfloor\frac{p+q}{2}\right\rfloor-1}\right\}, & \hspace*{0.1cm} \text{for } l=1 \vspace{0.1cm}\\ 
						\left\{v_{\left\lfloor\frac{p+q}{2}\right\rfloor},\cdots, v_{p-\left\lfloor\frac{p-q}{2}\right\rfloor}\right\}, & \hspace*{0.1cm} \text{for } l=2 \vspace{0.1cm}\\
						\left\{v_{p+1-\left\lfloor\frac{p-q}{2}\right\rfloor},\cdots, v_{p}\right\}, & \hspace*{0.1cm} \text{for } l=3 \vspace{0.1cm}\\  
						\left\{v_{p+1}\right\}, & \hspace*{0.1cm} \text{for } l=4 \vspace{0.1cm}\\ 
						\left\{v_{p+2},\cdots,v_{p+q}\right\}, & \hspace*{0.1cm} \text{for } l=5 \vspace{0.1cm}\\ 
						\left\{v_{p+q+1},\cdots,v_{p+q+r}\right\}, & \hspace*{0.1cm} \text{for } l=6. \\
					\end{cases}
				\end{equation*}
			\end{minipage}%
			\begin{minipage}{.5\linewidth}
				\begin{equation*}
					\RR(v_{\A} | \R)=\begin{cases}
						\left({\A}-1, {\A}+1\right), & \hspace*{0.1cm} \text{for } {\A}\in \V_1 \vspace{0.1cm}\\
						\left({\A}-1, p+q-({\A}+1)\right), & \hspace*{0.1cm} \text{for } {\A}\in \V_2 \vspace{0.1cm}\\
						
						\left(p+q+1-{\A}, p+q-({\A}+1)\right), & \hspace*{0.1cm} \text{for } {\A}\in \V_3 \vspace{0.1cm}\\
						
						\left({\A}-p,p+2-{\A}\right), & \hspace*{0.1cm} \text{for } {\A}\in \V_4 \vspace{0.1cm}\\
						
						\left({\A}-p,{\A}-(p+2)\right), & \hspace*{0.1cm} \text{for } {\A}\in \V_5 \vspace{0.1cm}\\
						
						\left({\A}+1-(p+q), {\A}+1-(p+q)\right), & \hspace*{0.1cm} \text{for } {\A}\in \V_6 \vspace{0.1cm}.\\
						
					\end{cases}
				\end{equation*}
			\end{minipage}
		}
		
		\vspace{0.2cm}

	\end{proofpart}
	%\noindent\rule[0.5ex]{\linewidth}{1pt}
	
	\begin{proofpart}
		\mybox{When $q-r>2$.}\label{part IX}
		Let $\V_l$ and their representations from $\W=\left\{v_1, v_{\left\lfloor\frac{p+r}{2}\right\rfloor+1}\right\}$, be as given in the following:

		\resizebox{9.5cm}{!}{%
			\noindent\begin{minipage}{.65\linewidth}
				\begin{equation*}
					\V_l=\begin{cases}
						\left\{v_1,v_2,\cdots, v_{\left\lfloor\frac{p+r}{2}\right\rfloor+1}\right\}, & \hspace*{0.1cm} \text{for } l=1 \vspace{0.1cm}\\ 
						\left\{ v_{\left\lfloor\frac{p+r}{2}\right\rfloor+2},\cdots, v_{p+1-\left\lfloor\frac{p-r}{2}\right\rfloor}\right\}, & \hspace*{0.1cm} \text{for } l=2 \vspace{0.1cm}\\
						\left\{v_{p+2-\left\lfloor\frac{p-r}{2}\right\rfloor}, \cdots, v_{p}\right\}, & \hspace*{0.1cm} \text{for } l=3 \vspace{0.1cm}\\
						\left\{v_{p+1},\cdots, v_{p+\left\lfloor\frac{q-r}{2}\right\rfloor}\right\}, & \hspace*{0.1cm} \text{for } l=4 \vspace{0.1cm}\\
						\left\{v_{p+1+\left\lfloor\frac{q-r}{2}\right\rfloor},\cdots, v_{p+q-\left\lfloor\frac{q-r}{2}\right\rfloor}\right\}, & \hspace*{0.1cm} \text{for } l=5 \\
						\left\{v_{p+q+1-\left\lfloor\frac{q-r}{2}\right\rfloor},\cdots, v_{p+q}\right\}, & \hspace*{0.1cm} \text{for } l=6 \\
						\left\{v_{p+q+1},\cdots,v_{p+q+r}\right\}, & \hspace*{0.1cm} \text{for } l=7. \\
						
					\end{cases}
				\end{equation*}
			\end{minipage}%
			\begin{minipage}{.5\linewidth}
				\begin{equation*}
					\RR(v_{\A} | \R)=\begin{cases}
						\left({\A}-1, \left\lfloor\frac{p+r}{2}\right\rfloor+1-{\A}\right), & \hspace*{0.1cm} \text{for } {\A}\in \V_1 \vspace{0.1cm}\\
						
						\left({\A}-1,{\A}-\left\lfloor\frac{p+r}{2}\right\rfloor-1\right), & \hspace*{0.1cm} \text{for } {\A}\in \V_2 \vspace{0.1cm}\\
						
						\left(p+r+3-{\A},{\A}-\left\lfloor\frac{p+r}{2}\right\rfloor-1\right), & \hspace*{0.1cm} \text{for } {\A}\in \V_3 \vspace{0.1cm}\\
						
						\left({\A}-p,{\A}+\left\lfloor\frac{p+r}{2}\right\rfloor-p\right), & \hspace*{0.1cm} \text{for } {\A}\in \V_4 \vspace{0.1cm}\\
						\left({\A}-p,2p+q-\left\lfloor\frac{p+r}{2}\right\rfloor-{\A}\right), & \hspace*{0.1cm} \text{for } {\A}\in \V_5 \vspace{0.1cm}\\
						\left(p+q+r+2-{\A},2p+q-\left\lfloor\frac{p+r}{2}\right\rfloor-{\A}\right), & \hspace*{0.1cm} \text{for } {\A}\in \V_6 \vspace{0.1cm}\\
						
						\left({\A}+1-(p+q), 2p+q+r+1-\left\lfloor\frac{p+r}{2}\right\rfloor-{\A}\right), & \hspace*{0.1cm} \text{for } {\A}\in \V_7 \vspace{0.1cm}.\\
					\end{cases}
				\end{equation*}
			\end{minipage}
		}
		
		\vspace{0.2cm}
		
		It can easily be verified once more that $\W$ is indeed a resolving set in this case.
		
		Finally, these three parts together with Theorem \ref{basic}-(c) give us the result.\qedhere

	\end{proofpart}
\end{proof}

\begin{Theorem}\label{Theorem 4}
	Let $\G=\C_{p,q,r}$ with $p=r$, and $q\lesseqqgtr r$, then
	
	\[
	\beta(\G)=\begin{cases}
		3 & \text{ when }q-r = 2 \text{ or } 4 \\
		2 & \text{otherwise}. \\
	\end{cases}
	\]
	
\end{Theorem}
\begin{proof}
	Let $\W$ be as given in the following:
	
	\begin{equation*}
		\W=\begin{cases}
			\left\{v_1, v_2,v_{\left\lfloor\frac{q-r}{2}\right\rfloor+p+1}\right\} & \hspace*{0.1cm} \text{for } q-r=2 \text{ or }4 \vspace{0.1cm},\\
			\left\{v_1,v_{\left\lfloor\frac{p+q}{2}\right\rfloor}\right\} & \hspace*{0.1cm} \text{for } q-r < 2 \vspace{0.1cm}\\ 
			\left\{v_1, v_{p+2}\right\} & \hspace*{0.1cm} \text{for } q-r>2 \vspace{0.1cm}.\\
			
		\end{cases}
	\end{equation*}
	
	%\noindent\rule[0.5ex]{\linewidth}{1pt}
	
	\begin{proofpart}
		\mybox{When $q-r=2$ or $4$.}
		
		Let us partition $\V(\G)$ into $\V_l$ as given in the following, along with their representation from $\W=\left\{v_1, v_2,v_{\left\lfloor\frac{q-r}{2}\right\rfloor+p+1}\right\}$:
		
		\resizebox{9cm}{!}{%
			\noindent\begin{minipage}{.7\linewidth}
				\begin{equation*}
					\V_l=\begin{cases}
						\left\{v_1\right\}, & \hspace*{0.1cm} \text{for } l=1 \vspace{0.1cm}\\ 
						\left\{v_2,v_3,\cdots,v_{p}\right\}, & \hspace*{0.1cm} \text{for } l=2 \vspace{0.1cm}\\
						\left\{v_{p+1},\cdots,v_{p+1+\left\lfloor\frac{q-p}{2}\right\rfloor}\right\} & \hspace*{0.1cm} \text{for } l=3 \vspace{0.1cm}\\ 
						\left\{v_{p+2+\left\lfloor\frac{q-p}{2}\right\rfloor},\cdots,v_{p+q-1-\left\lfloor\frac{q-p}{2}\right\rfloor}\right\}, & \hspace*{0.1cm} \text{for } l=4 \vspace{0.1cm}\\ 
						\left\{v_{p+q-\left\lfloor\frac{q-p}{2}\right\rfloor},\cdots, v_{p+q}\right\}, & \hspace*{0.1cm} \text{for } l=5 \vspace{0.1cm}\\
						
						\left\{v_{p+q+1},\cdots,v_{p+q+r-1}\right\}, & \hspace*{0.1cm} \text{for } l=6 \\
						\left\{v_{p+q+r}\right\}, & \hspace*{0.1cm} \text{for } l=7 .\vspace{0.1cm}\\ 
					\end{cases}
				\end{equation*}
			\end{minipage}%
			\begin{minipage}{.5\linewidth}
				\begin{equation*}
					\RR(v_{\A} | \R)=\begin{cases}
						\left({\A}-1, 2-{\A}, {\A}+\left\lfloor\frac{q-p}{2}\right\rfloor\right), & \hspace*{0.1cm} \text{for } {\A}\in \V_1 \vspace{0.1cm}\\
						
						\left({\A}-1, {\A}-2, {\A}+\left\lfloor\frac{q-p}{2}\right\rfloor\right), & \hspace*{0.1cm} \text{for } {\A}\in \V_2 \vspace{0.1cm}\\
						\left({\A}-p, {\A}+1-p,  p+1+\left\lfloor\frac{q-p}{2}\right\rfloor-{\A}\right), & \hspace*{0.1cm} \text{for } {\A}\in \V_3 \vspace{0.1cm}\\
						
						\left({\A}-p,{\A}+1-p, {\A}-(p+1+\left\lfloor\frac{q-p}{2}\right\rfloor)\right), & \hspace*{0.1cm} \text{for } {\A}\in \V_4 \vspace{0.1cm}\\
						\left(2p+q-{\A},2p+q-1-{\A},{\A}-(p+1+\left\lfloor\frac{q-p}{2}\right\rfloor)\right), & \hspace*{0.1cm} \text{for } {\A}\in \V_5 \vspace{0.1cm}\\
						\left({\A}+1-(p+q),{\A}+2-(p+q),{\A}+\left\lfloor\frac{q-p}{2}\right\rfloor-(p+q)\right), & \hspace*{0.1cm} \text{for } {\A}\in \V_6 \vspace{0.1cm}\\
						\left({\A}+1-(p+q),{\A}-(p+q),{\A}+\left\lfloor\frac{q-p}{2}\right\rfloor-(p+q)\right), & \hspace*{0.1cm} \text{for } {\A}\in \V_7 \vspace{0.1cm}.\\

					\end{cases}
				\end{equation*}
			\end{minipage}
		}
		
		\vspace{0.2cm}
		
		It can be easily proved that $\W$ is a resolving set, by considering two distinct vertices having same representation, and showing that this generates a contradiction. To show that $\W=\left\{v_1, v_2,v_{\left\lfloor\frac{q-r}{2}\right\rfloor+p+1}\right\}$ is a minimal resolving set, we consider the following three cases.
		
		\begin{mycases}
			\case When we remove $v_1$ from $\W$.
			
			Let $\W'=\W-v_1=\{v_2,v_{\left\lfloor\frac{q-r}{2}\right\rfloor+p+1}\}$. Let us consider the vertices $v_{p+q-1}$ and $v_{p+q+r-2}$ from Figure \ref{typeiiibicyclic}. Using the representations given above, it can be easily shown that $\RR(v_{p+q}|\W)=\RR(v_{p+q+r-1}|\W)$.
			
			\case When we remove $v_2$ from $\W$.
			
			Let $\W'=\W-v_2=\{v_1,v_{\left\lfloor\frac{q-r}{2}\right\rfloor+p+1}\}$. Let us consider the vertices $v_{p+q}$ and $v_{p+q+r-1}$. Again, using the representations given above, it can be easily proved that $\RR(v_{p}|\W)=\RR(v_{p+q+r}|\W)$.
			
			\case When we remove $v_{\left\lfloor\frac{q-r}{2}\right\rfloor+p+1}$ from $\W$.
			
			Let $\W'=\W-v_{\left\lfloor\frac{q-r}{2}\right\rfloor+p+1}=\{v_1,v_2\}$. Considering the vertices $v_{p+q-1}$ and $v_{p+q+r}$, one can see that they have same representation with respect to $\W'$.
		\end{mycases}
		
		These cases ensure that $\W$ is a smallest resolving set and hence, $\beta(\G)=3$ for this case.

	\end{proofpart}
	
	\begin{proofpart}
		\mybox{When $q-r>2$ and $q-r \neq 4$.}\label{part X}
		
		Let $\V_l$ and their representations from $\W=\{v_1,v_{p+2}\}$ be as given in the following:
		
		\resizebox{11cm}{!}{%
			\noindent\begin{minipage}{.7\linewidth}
				\begin{equation*}
					\V_l=\begin{cases}
						\left\{v_1,\cdots,v_{p}\right\}, & \hspace*{0.1cm} \text{for } l=1 \vspace{0.1cm}\\ 
						\left\{v_{p+1}\right\}, & \hspace*{0.1cm} \text{for } l=2 \vspace{0.1cm}\\
						\left\{v_{p+2},\cdots, v_{p+q-\left\lceil\frac{q-r}{2}\right\rceil}\right\}, & \hspace*{0.1cm} \text{for } l=3 \vspace{0.1cm}\\  
						
						\left\{v_{p+q+1-\left\lceil\frac{q-r}{2}\right\rceil} \cdots,v_{2p+2+\left\lceil\frac{q-p-2}{2}\right\rceil}\right\}, & \hspace*{0.1cm} \text{for } l=4 \vspace{0.1cm}\\ 
						\left\{v_{2p+3+\left\lceil\frac{q-p-2}{2}\right\rceil},\cdots,v_{p+q}\right\}, & \hspace*{0.1cm} \text{for } l=5 \vspace{0.1cm}\\ 
						\left\{v_{p+q+1},\cdots,v_{p+q+r}\right\}, & \hspace*{0.1cm} \text{for } l=6. \\
					\end{cases}
				\end{equation*}
			\end{minipage}%
			\begin{minipage}{.5\linewidth}
				\begin{equation*}
					\RR(v_{\A} | \R)=\begin{cases}
						\left({\A}-1, {\A}+1\right), & \hspace*{0.1cm} \text{for } {\A}\in \V_1 \vspace{0.1cm}\\
						\left({\A}-p, p+2-{\A}\right), & \hspace*{0.1cm} \text{for } {\A}\in \V_2 \vspace{0.1cm}\\
						
						\left({\A}-p, {\A}-(p+2)\right), & \hspace*{0.1cm} \text{for } {\A}\in \V_3 \vspace{0.1cm}\\
						
						\left(2p+q-{\A},{\A}-(p+2)\right), & \hspace*{0.1cm} \text{for } {\A}\in \V_4 \vspace{0.1cm}\\
						
						\left(2p+q-{\A},2p+q+2-{\A}\right), & \hspace*{0.1cm} \text{for } {\A}\in \V_5 \vspace{0.1cm}\\
						
						\left({\A}+1-(p+q), {\A}+1-(p+q)\right), & \hspace*{0.1cm} \text{for } {\A}\in \V_6 \vspace{0.1cm}.\\
						
					\end{cases}
				\end{equation*}
			\end{minipage}
		}
		
		\vspace{0.2cm}

	\end{proofpart}
	%\noindent\rule[0.5ex]{\linewidth}{1pt}
	
	\begin{proofpart}
		\mybox{When $q-r<2$.}\label{part XI}
		
		We subdivide this part into two cases.
		
		\begin{mycases}
			
			\case When $q-r=1$.
			
			Let $\V_l$ and their representations, from $\W=\left\{v_1,v_{\left\lfloor\frac{p+q}{2}\right\rfloor}\right\}$, be provided as in the following:
			
			\resizebox{10.5cm}{!}{%
				\noindent\begin{minipage}{.55\linewidth}
					\begin{equation*}
						\V_l=\begin{cases}
							
							\left\{v_1,v_2,v_3,\cdots,v_{p}\right\}, & \hspace*{0.1cm} \text{for } l=1 \vspace{0.1cm}\\
							\left\{v_{p+1}\right\}, & \hspace*{0.1cm} \text{for } l=2 \vspace{0.1cm}\\ 
							\left\{v_{p+2},\cdots,v_{p+q-1}\right\} & \hspace*{0.1cm} \text{for } l=3 \vspace{0.1cm}\\ 
							\left\{v_{p+q}\right\}, & \hspace*{0.1cm} \text{for } l=4 \vspace{0.1cm}\\
							\left\{v_{p+q+1},\cdots,v_{p+q+r}\right\}, & \hspace*{0.1cm} \text{for } l=5. \vspace{0.1cm}\\ 
						\end{cases}
					\end{equation*}
				\end{minipage}%
				\begin{minipage}{.5\linewidth}
					\begin{equation*}
						\RR(v_{\A} | \R)=\begin{cases}
							\left({\A}-1,p-{\A}\right), & \hspace*{0.1cm} \text{for } {\A}\in \V_1 \vspace{0.1cm}\\
							
							\left({\A}-p, p\right), & \hspace*{0.1cm} \text{for } {\A}\in \V_2 \vspace{0.1cm}\\
							
							\left({\A}-p, p+q+1-{\A}\right), & \hspace*{0.1cm} \text{for } {\A}\in \V_3 \vspace{0.1cm}\\
							
							\left(p,p+q+1-{\A}\right), & \hspace*{0.1cm} \text{for } {\A}\in \V_4 \\
							
							\left({\A}+1-(p+q),p+q+r+2-{\A}\right), & \hspace*{0.1cm} \text{for } {\A}\in \V_5 \vspace{0.1cm}.\\						
						\end{cases}
					\end{equation*}
				\end{minipage}
			}
			
			\vspace{0.2cm}

			\case {When $q-r<1$.} \label{part XII}
			
			For this case, let $V_l$ and the representations be as provided below.
			
			\resizebox{11cm}{!}{%
				\noindent\begin{minipage}{.7\linewidth}
					\begin{equation*}
						\V_l=\begin{cases}
							\left\{v_1,\cdots,v_{\left\lfloor\frac{p+q}{2}\right\rfloor}\right\}, & \hspace*{0.1cm} \text{for } l=1 \vspace{0.1cm}\\ 
							\left\{v_{\left\lfloor\frac{p+q}{2}\right\rfloor+1},\cdots, v_{p-\left\lfloor\frac{p-q}{2}\right\rfloor}\right\}, & \hspace*{0.1cm} \text{for } l=2 \vspace{0.1cm}\\
							\left\{v_{p+1-\left\lfloor\frac{p-q}{2}\right\rfloor},\cdots, v_{p}\right\}, & \hspace*{0.1cm} \text{for } l=3 \vspace{0.1cm}\\  
							
							\left\{v_{p+1}\right\}, & \hspace*{0.1cm} \text{for } l=4 \vspace{0.1cm}\\ 
							\left\{v_{p+2},\cdots,v_{p+q}\right\}, & \hspace*{0.1cm} \text{for } l=5 \vspace{0.1cm}\\ 
							\left\{v_{p+q+1},\cdots v_{p+q+1+\left\lceil\frac{r-q}{2}\right\rceil}\right\}, & \hspace*{0.1cm} \text{for } l=6 \vspace{0.1cm}\\
							
							\left\{v_{p+q+2+\left\lceil\frac{r-q}{2}\right\rceil},\cdots,v_{p+q+m-\left\lceil\frac{r-q}{2}\right\rceil}\right\}, & \hspace*{0.1cm} \text{for } l=7 \\
							\left\{v_{p+q+r+1-\left\lceil\frac{r-q}{2}\right\rceil},\cdots,v_{p+q+r}\right\}, & \hspace*{0.1cm} \text{for } l=8. \\
						\end{cases}
					\end{equation*}
				\end{minipage}%
				\begin{minipage}{.7\linewidth}
					\begin{equation*}
						\RR(v_{\A} | \R)=\begin{cases}
							\left({\A}-1, \left\lfloor\frac{p+q}{2}\right\rfloor-{\A}\right), & \hspace*{0.1cm} \text{for } {\A}\in \V_1 \vspace{0.1cm}\\
							\left({\A}-1, {\A}-\left\lfloor\frac{p+q}{2}\right\rfloor\right), & \hspace*{0.1cm} \text{for } {\A}\in \V_2 \vspace{0.1cm}\\
							
							\left(p+q+1-{\A}, {\A}-\left\lfloor\frac{p+q}{2}\right\rfloor\right), & \hspace*{0.1cm} \text{for } {\A}\in \V_3 \vspace{0.1cm}\\
							
							\left({\A}-p,\left\lfloor\frac{p+q}{2}\right\rfloor\right), & \hspace*{0.1cm} \text{for } {\A}\in \V_4 \vspace{0.1cm}\\
							
							\left({\A}-p,p+q+r+1-\left\lfloor\frac{p+q}{2}\right\rfloor-{\A}\right), & \hspace*{0.1cm} \text{for } {\A}\in \V_5 \vspace{0.1cm}\\
							
							\left({\A}+1-(p+q), {\A}+\left\lfloor\frac{p+q}{2}\right\rfloor-(p+q)\right), & \hspace*{0.1cm} \text{for } {\A}\in \V_6 \vspace{0.1cm}\\
							\left ({\A}+1-(p+q),2p+q+r+2-({\A}+\left\lfloor\frac{p+q}{2}\right\rfloor)\right), & \hspace*{0.1cm} \text{for } {\A}\in \V_7 \vspace{0.1cm}\\
							\left(p+2q+r+1-{\A},2p+q+r+2-({\A}+\left\lfloor\frac{p+q}{2}\right\rfloor)\right), & \hspace*{0.1cm} \text{for } {\A}\in \V_8. \vspace{0.1cm}\\
						\end{cases}
					\end{equation*}
				\end{minipage}
			}
			
			\vspace{0.2cm}

	\end{mycases}
		
		For both these cases, it is again a trivial matter to show that $\W$ is a resolving set.
	\end{proofpart}
	Parts 2 and 3 with Theorem \ref{basic}-(c) enable us to conclude that $\beta(\G)=2$ for both these parts.
\end{proof}

\begin{Corollary}
	Let $\G$ be a $\Theta$-graph then $\beta(\Theta_{p,p+2,p})=3$.
\end{Corollary}
\begin{proof}
	By Theorem \ref{Theorem 4}, Part 1, $\beta(\C_{p,q,r})=3$ when $q-r=2$ or $4$. Since $p=r$ and $\C_{p,q,r}=\Theta_{p+1,q-1,r+1}$, we get our result.
\end{proof}

\section{Optimizing agricultural supply chain logistics with theoretical
	application}

In a world where supply chains are becoming increasingly complex and global, the application of graph theory is indispensable for maintaining a competitive edge and ensuring uninterrupted operations. In supply chain networks, nodes can represent various entities such as suppliers, manufacturers, distribution centers, and retailers. Edges can represent the transportation routes or the flow of goods between these entities. In this theoretical exploration, we tackle the problem of optimizing agricultural supply chain logistics within a diverse network of fields. While the practical implications of this problem are rooted in logistics and operations management, we approach it through the abstract nature of graph theory. This approach allows us to leverage mathematical structures and concepts to model and solve complex logistical challenges in agriculture.

Imagine managing a large agricultural enterprise spread across a region with diverse landscapes, including rolling hills, dense forests, and open plains. The farm consists of fields labeled Field 1, Field 2 up to Field 12. Every Field presents a unique challenges for transporting harvested crops to processing facilities. Our goal is to devise an efficient strategy to manage the logistics of transporting these crops using two types of transportation equipment: ground-based trucks (Super Truck) and aerial drones (Agri Drone).

Let us suppose that Fields 1 to 8 are accessible by ground trucks, but the uneven terrain and dense vegetation in some areas complicate their use. On the other hand, Fields 6 to 12 are more suitable for aerial drones, especially where ground trucks are hindered by topographical obstacles. The aerial drones are less effective in fields with significant tree cover or close proximity to water bodies where precision is critical. Ground trucks, while effective in open fields, struggle with accessibility in steep or forested areas. To manage these challenges, we represent the fields and their availability of transportation methods by using two overlapping cycles:

Ground Truck Cycle $\mathcal{C}_{1}:\{\text{Field 1}, \text{Field 2}, \cdots, \text{Field 8}\};$

Aerial Drone Cycle $\mathcal{C}_{2}:\{\text{Field 6}, \text{Field 7}, \cdots, \text{Field 12}\}.$

The fields $\{\text{Field 6}, \text{Field 7}, \text{Field 8}\}$ can be serviced by both types of transporters, creating a complex network where strategic planning is crucial to avoid inefficiencies. If we represent this whole network using one connected graph, it can be easily modeled using the following graphical structure..

\begin{figure}[H]
	\centering
	\tikzset{every picture/.style={line width=0.75pt}} %set default line width to 0.75pt        
	\resizebox{10cm}{4.8cm}{%
		\begin{tikzpicture}[x=0.75pt,y=0.75pt,yscale=-1,xscale=1]
			%uncomment if require: \path (0,300); %set diagram left start at 0, and has height of 300

	%Shape: Circle [id:dp743632011458764] 
	\draw  [fill={rgb, 255:red, 126; green, 211; blue, 33 }  ,fill opacity=1 ] (201,39.92) .. controls (201,30.02) and (209.02,22) .. (218.92,22) .. controls (228.81,22) and (236.83,30.02) .. (236.83,39.92) .. controls (236.83,49.81) and (228.81,57.83) .. (218.92,57.83) .. controls (209.02,57.83) and (201,49.81) .. (201,39.92) -- cycle ;
	%Curve Lines [id:da25535526422438437] 
	\draw    (236.83,39.92) .. controls (276.83,9.92) and (282.83,81) .. (322.83,51) ;
	%Rounded Rect [id:dp15652296360057294] 
	\draw  [fill={rgb, 255:red, 126; green, 211; blue, 33 }  ,fill opacity=1 ] (322.83,35.4) .. controls (322.83,32.53) and (325.16,30.2) .. (328.03,30.2) -- (356.47,30.2) .. controls (359.34,30.2) and (361.67,32.53) .. (361.67,35.4) -- (361.67,51) .. controls (361.67,53.87) and (359.34,56.2) .. (356.47,56.2) -- (328.03,56.2) .. controls (325.16,56.2) and (322.83,53.87) .. (322.83,51) -- cycle ;
	%Shape: Ellipse [id:dp603825005744679] 
	\draw  [fill={rgb, 255:red, 126; green, 211; blue, 33 }  ,fill opacity=1 ] (426.83,37.4) .. controls (426.83,29.39) and (437.32,22.9) .. (450.25,22.9) .. controls (463.18,22.9) and (473.67,29.39) .. (473.67,37.4) .. controls (473.67,45.41) and (463.18,51.9) .. (450.25,51.9) .. controls (437.32,51.9) and (426.83,45.41) .. (426.83,37.4) -- cycle ;
	%Curve Lines [id:da13990640417751066] 
	\draw    (361.67,35.4) .. controls (395.83,61) and (391.83,22) .. (426.83,37.4) ;
	%Shape: Ellipse [id:dp6977576052505512] 
	\draw  [fill={rgb, 255:red, 126; green, 211; blue, 33 }  ,fill opacity=1 ] (508.42,90.5) .. controls (508.42,82.49) and (518.9,76) .. (531.83,76) .. controls (544.77,76) and (555.25,82.49) .. (555.25,90.5) .. controls (555.25,98.51) and (544.77,105) .. (531.83,105) .. controls (518.9,105) and (508.42,98.51) .. (508.42,90.5) -- cycle ;
	%Curve Lines [id:da06426022661675135] 
	\draw    (473.67,37.4) .. controls (513.83,26) and (533.83,45) .. (531.83,76) ;
	%Rounded Rect [id:dp4662600847287661] 
	\draw  [fill={rgb, 255:red, 126; green, 211; blue, 33 }  ,fill opacity=1 ] (102.63,80.2) .. controls (102.63,77.33) and (104.96,75) .. (107.83,75) -- (136.27,75) .. controls (139.14,75) and (141.47,77.33) .. (141.47,80.2) -- (141.47,95.8) .. controls (141.47,98.67) and (139.14,101) .. (136.27,101) -- (107.83,101) .. controls (104.96,101) and (102.63,98.67) .. (102.63,95.8) -- cycle ;
	%Curve Lines [id:da06506760096992315] 
	\draw    (141.47,80.2) .. controls (146.47,53.2) and (167.83,36) .. (201,39.92) ;
	%Curve Lines [id:da8430782053184231] 
	\draw    (141.47,95.8) .. controls (147.83,117) and (171.83,133) .. (199.83,128) ;
	%Shape: Circle [id:dp3518894963670862] 
	\draw  [fill={rgb, 255:red, 126; green, 211; blue, 33 }  ,fill opacity=1 ] (199.83,128) .. controls (199.83,118.1) and (207.85,110.08) .. (217.75,110.08) .. controls (227.65,110.08) and (235.67,118.1) .. (235.67,128) .. controls (235.67,137.9) and (227.65,145.92) .. (217.75,145.92) .. controls (207.85,145.92) and (199.83,137.9) .. (199.83,128) -- cycle ;
	%Shape: Ellipse [id:dp389647718721555] 
	\draw  [fill={rgb, 255:red, 126; green, 211; blue, 33 }  ,fill opacity=1 ] (300.83,130) .. controls (300.83,121.99) and (311.32,115.5) .. (324.25,115.5) .. controls (337.18,115.5) and (347.67,121.99) .. (347.67,130) .. controls (347.67,138.01) and (337.18,144.5) .. (324.25,144.5) .. controls (311.32,144.5) and (300.83,138.01) .. (300.83,130) -- cycle ;
	%Curve Lines [id:da8308816599914775] 
	\draw    (235.67,128) .. controls (269.83,153.6) and (265.83,114.6) .. (300.83,130) ;
	%Curve Lines [id:da43063010386054] 
	\draw    (347.67,130) .. controls (387.67,100) and (393.67,171.08) .. (433.67,141.08) ;
	%Rounded Rect [id:dp4289533783681887] 
	\draw  [fill={rgb, 255:red, 126; green, 211; blue, 33 }  ,fill opacity=1 ] (433.67,125.48) .. controls (433.67,122.61) and (435.99,120.28) .. (438.87,120.28) -- (467.3,120.28) .. controls (470.17,120.28) and (472.5,122.61) .. (472.5,125.48) -- (472.5,141.08) .. controls (472.5,143.96) and (470.17,146.28) .. (467.3,146.28) -- (438.87,146.28) .. controls (435.99,146.28) and (433.67,143.96) .. (433.67,141.08) -- cycle ;
	%Curve Lines [id:da8439056791755848] 
	\draw    (531.83,105) .. controls (528.83,130) and (499.83,143) .. (472.5,125.48) ;
	%Shape: Ellipse [id:dp5138734804718783] 
	\draw  [fill={rgb, 255:red, 126; green, 211; blue, 33 }  ,fill opacity=1 ] (151.83,219) .. controls (151.83,210.99) and (162.32,204.5) .. (175.25,204.5) .. controls (188.18,204.5) and (198.67,210.99) .. (198.67,219) .. controls (198.67,227.01) and (188.18,233.5) .. (175.25,233.5) .. controls (162.32,233.5) and (151.83,227.01) .. (151.83,219) -- cycle ;
	%Curve Lines [id:da4774588708456571] 
	\draw    (175.25,204.5) .. controls (217.83,177) and (163.83,168) .. (217.75,145.92) ;
	%Shape: Circle [id:dp2766034873811065] 
	\draw  [fill={rgb, 255:red, 126; green, 211; blue, 33 }  ,fill opacity=1 ] (499,221.92) .. controls (499,212.02) and (507.02,204) .. (516.92,204) .. controls (526.81,204) and (534.83,212.02) .. (534.83,221.92) .. controls (534.83,231.81) and (526.81,239.83) .. (516.92,239.83) .. controls (507.02,239.83) and (499,231.81) .. (499,221.92) -- cycle ;
	%Curve Lines [id:da28093480789203196] 
	\draw    (467.3,146.28) .. controls (517.83,156) and (476.83,188) .. (516.92,204) ;
	%Rounded Rect [id:dp723902025147352] 
	\draw  [fill={rgb, 255:red, 126; green, 211; blue, 33 }  ,fill opacity=1 ] (276.83,210.4) .. controls (276.83,207.53) and (279.16,205.2) .. (282.03,205.2) -- (310.47,205.2) .. controls (313.34,205.2) and (315.67,207.53) .. (315.67,210.4) -- (315.67,226) .. controls (315.67,228.87) and (313.34,231.2) .. (310.47,231.2) -- (282.03,231.2) .. controls (279.16,231.2) and (276.83,228.87) .. (276.83,226) -- cycle ;
	%Curve Lines [id:da45562312686535367] 
	\draw    (198.67,219) .. controls (238.67,189) and (242.03,261.2) .. (282.03,231.2) ;
	%Shape: Ellipse [id:dp9189856540682646] 
	\draw  [fill={rgb, 255:red, 126; green, 211; blue, 33 }  ,fill opacity=1 ] (388.83,219) .. controls (388.83,210.99) and (399.32,204.5) .. (412.25,204.5) .. controls (425.18,204.5) and (435.67,210.99) .. (435.67,219) .. controls (435.67,227.01) and (425.18,233.5) .. (412.25,233.5) .. controls (399.32,233.5) and (388.83,227.01) .. (388.83,219) -- cycle ;
	%Curve Lines [id:da5734105125628339] 
	\draw    (435.67,219) .. controls (475.67,189) and (459,251.92) .. (499,221.92) ;
	%Curve Lines [id:da5069862137565662] 
	\draw    (315.67,226) .. controls (333.83,242) and (370.83,243) .. (388.83,219) ;
	
	% Text Node
	\draw (64,61) node [anchor=north west][inner sep=0.75pt]  [font=\scriptsize] [align=left] {Field 1};
	% Text Node
	\draw (204,8) node [anchor=north west][inner sep=0.75pt]  [font=\scriptsize] [align=left] {Field 2};
	% Text Node
	\draw (324,9) node [anchor=north west][inner sep=0.75pt]  [font=\scriptsize] [align=left] {Field 3};
	% Text Node
	\draw (436,8) node [anchor=north west][inner sep=0.75pt]  [font=\scriptsize] [align=left] {Field 4};
	% Text Node
	\draw (544,60) node [anchor=north west][inner sep=0.75pt]  [font=\scriptsize] [align=left] {Field 5};
	% Text Node
	\draw (204,91) node [anchor=north west][inner sep=0.75pt]  [font=\scriptsize] [align=left] {Field 6};
	% Text Node
	\draw (308,90) node [anchor=north west][inner sep=0.75pt]  [font=\scriptsize] [align=left] {Field 7};
	% Text Node
	\draw (436,91) node [anchor=north west][inner sep=0.75pt]  [font=\scriptsize] [align=left] {Field 8};
	% Text Node
	\draw (159,250) node [anchor=north west][inner sep=0.75pt]  [font=\scriptsize] [align=left] {Field 9};
	% Text Node
	\draw (277,250) node [anchor=north west][inner sep=0.75pt]  [font=\scriptsize] [align=left] {Field 10};
	% Text Node
	\draw (395,250) node [anchor=north west][inner sep=0.75pt]  [font=\scriptsize] [align=left] {Field 11};
	% Text Node
	\draw (501,250) node [anchor=north west][inner sep=0.75pt]  [font=\scriptsize] [align=left] {Field 12};
	% Text Node
	\draw (109.83,78.4) node [anchor=north west][inner sep=0.75pt]    {$v_{1}$};
	% Text Node
	\draw (210.83,30.4) node [anchor=north west][inner sep=0.75pt]    {$v_{2}$};
	% Text Node
	\draw (330.03,33.6) node [anchor=north west][inner sep=0.75pt]    {$v_{3}$};
	% Text Node
	\draw (441.83,27.4) node [anchor=north west][inner sep=0.75pt]    {$v_{4}$};
	% Text Node
	\draw (522.83,81.4) node [anchor=north west][inner sep=0.75pt]    {$v_{5}$};
	% Text Node
	\draw (208.83,119.4) node [anchor=north west][inner sep=0.75pt]    {$v_{6}$};
	% Text Node
	\draw (316.83,118.4) node [anchor=north west][inner sep=0.75pt]    {$v_{7}$};
	% Text Node
	\draw (440.87,123.68) node [anchor=north west][inner sep=0.75pt]    {$v_{8}$};
	% Text Node
	\draw (166.83,208.4) node [anchor=north west][inner sep=0.75pt]    {$v_{9}$};
	% Text Node
	\draw (284.03,208.6) node [anchor=north west][inner sep=0.75pt]    {$v_{10}$};
	% Text Node
	\draw (402.03,209.6) node [anchor=north west][inner sep=0.75pt]    {$v_{11}$};
	% Text Node
	\draw (506.03,213.6) node [anchor=north west][inner sep=0.75pt]    {$v_{12}$};

\end{tikzpicture}
}
\caption{A Logistics Network of Fields Serviceable by Air and Land }
\label{application01}
\end{figure}
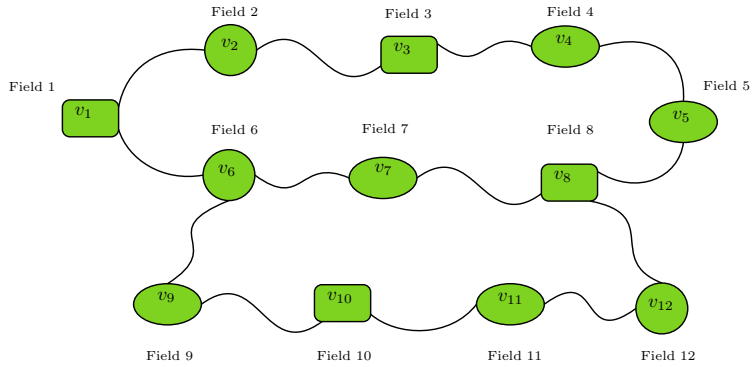

If we rename the fields 1 through 12 by $\{v_1,v_2, \cdots, v_{12}\}$, as shown in figure \ref{application01}, we transform this logistic network into a graph theoretical graph. This graph is given in the figure \ref{application}.

\begin{figure}[H]
	\centering
	\tikzset{every picture/.style={line width=0.75pt}} %set default line width to 0.75pt        
	\resizebox{5.4cm}{4.5cm}{%
		\begin{tikzpicture}[x=0.75pt,y=0.75pt,yscale=-1,xscale=1]
			%uncomment if require: \path (0,300); %set diagram left start at 0, and has height of 300
			
			%Shape: Ellipse [id:dp6751672508780109] 
			\draw  [fill={rgb, 255:red, 0; green, 0; blue, 0 }  ,fill opacity=1 ] (394.59,60.09) .. controls (394.59,57.24) and (396.79,54.93) .. (399.51,54.93) .. controls (402.23,54.93) and (404.43,57.24) .. (404.43,60.09) .. controls (404.43,62.94) and (402.23,65.25) .. (399.51,65.25) .. controls (396.79,65.25) and (394.59,62.94) .. (394.59,60.09) -- cycle ;
			%Shape: Ellipse [id:dp5549308006374756] 
			\draw  [fill={rgb, 255:red, 0; green, 0; blue, 0 }  ,fill opacity=1 ] (346.59,60.09) .. controls (346.59,57.24) and (348.79,54.93) .. (351.51,54.93) .. controls (354.23,54.93) and (356.43,57.24) .. (356.43,60.09) .. controls (356.43,62.94) and (354.23,65.25) .. (351.51,65.25) .. controls (348.79,65.25) and (346.59,62.94) .. (346.59,60.09) -- cycle ;
			%Shape: Ellipse [id:dp8236357940148391] 
			\draw  [fill={rgb, 255:red, 0; green, 0; blue, 0 }  ,fill opacity=1 ] (295.59,61.09) .. controls (295.59,58.24) and (297.79,55.93) .. (300.51,55.93) .. controls (303.23,55.93) and (305.43,58.24) .. (305.43,61.09) .. controls (305.43,63.94) and (303.23,66.25) .. (300.51,66.25) .. controls (297.79,66.25) and (295.59,63.94) .. (295.59,61.09) -- cycle ;
			%Shape: Ellipse [id:dp8991647347196112] 
			\draw  [fill={rgb, 255:red, 0; green, 0; blue, 0 }  ,fill opacity=1 ] (245.59,61.09) .. controls (245.59,58.24) and (247.79,55.93) .. (250.51,55.93) .. controls (253.23,55.93) and (255.43,58.24) .. (255.43,61.09) .. controls (255.43,63.94) and (253.23,66.25) .. (250.51,66.25) .. controls (247.79,66.25) and (245.59,63.94) .. (245.59,61.09) -- cycle ;
			%Shape: Ellipse [id:dp624972059604731] 
			\draw  [fill={rgb, 255:red, 0; green, 0; blue, 0 }  ,fill opacity=1 ] (195.59,61.09) .. controls (195.59,58.24) and (197.79,55.93) .. (200.51,55.93) .. controls (203.23,55.93) and (205.43,58.24) .. (205.43,61.09) .. controls (205.43,63.94) and (203.23,66.25) .. (200.51,66.25) .. controls (197.79,66.25) and (195.59,63.94) .. (195.59,61.09) -- cycle ;
			%Shape: Ellipse [id:dp7296911626222111] 
			\draw  [fill={rgb, 255:red, 0; green, 0; blue, 0 }  ,fill opacity=1 ] (346.59,126.09) .. controls (346.59,123.24) and (348.79,120.93) .. (351.51,120.93) .. controls (354.23,120.93) and (356.43,123.24) .. (356.43,126.09) .. controls (356.43,128.94) and (354.23,131.25) .. (351.51,131.25) .. controls (348.79,131.25) and (346.59,128.94) .. (346.59,126.09) -- cycle ;
			%Shape: Ellipse [id:dp2392475733401822] 
			\draw  [fill={rgb, 255:red, 0; green, 0; blue, 0 }  ,fill opacity=1 ] (295.59,127.09) .. controls (295.59,124.24) and (297.79,121.93) .. (300.51,121.93) .. controls (303.23,121.93) and (305.43,124.24) .. (305.43,127.09) .. controls (305.43,129.94) and (303.23,132.25) .. (300.51,132.25) .. controls (297.79,132.25) and (295.59,129.94) .. (295.59,127.09) -- cycle ;
			%Shape: Ellipse [id:dp29114661248337637] 
			\draw  [fill={rgb, 255:red, 0; green, 0; blue, 0 }  ,fill opacity=1 ] (245.59,127.09) .. controls (245.59,124.24) and (247.79,121.93) .. (250.51,121.93) .. controls (253.23,121.93) and (255.43,124.24) .. (255.43,127.09) .. controls (255.43,129.94) and (253.23,132.25) .. (250.51,132.25) .. controls (247.79,132.25) and (245.59,129.94) .. (245.59,127.09) -- cycle ;
			%Shape: Ellipse [id:dp7195114945845309] 
			\draw  [fill={rgb, 255:red, 0; green, 0; blue, 0 }  ,fill opacity=1 ] (372.59,187.09) .. controls (372.59,184.24) and (374.79,181.93) .. (377.51,181.93) .. controls (380.23,181.93) and (382.43,184.24) .. (382.43,187.09) .. controls (382.43,189.94) and (380.23,192.25) .. (377.51,192.25) .. controls (374.79,192.25) and (372.59,189.94) .. (372.59,187.09) -- cycle ;
			%Shape: Ellipse [id:dp39363791907628176] 
			\draw  [fill={rgb, 255:red, 0; green, 0; blue, 0 }  ,fill opacity=1 ] (321.59,188.09) .. controls (321.59,185.24) and (323.79,182.93) .. (326.51,182.93) .. controls (329.23,182.93) and (331.43,185.24) .. (331.43,188.09) .. controls (331.43,190.94) and (329.23,193.25) .. (326.51,193.25) .. controls (323.79,193.25) and (321.59,190.94) .. (321.59,188.09) -- cycle ;
			%Shape: Ellipse [id:dp9979431594327077] 
			\draw  [fill={rgb, 255:red, 0; green, 0; blue, 0 }  ,fill opacity=1 ] (271.59,188.09) .. controls (271.59,185.24) and (273.79,182.93) .. (276.51,182.93) .. controls (279.23,182.93) and (281.43,185.24) .. (281.43,188.09) .. controls (281.43,190.94) and (279.23,193.25) .. (276.51,193.25) .. controls (273.79,193.25) and (271.59,190.94) .. (271.59,188.09) -- cycle ;
			%Shape: Ellipse [id:dp0022301204411210307] 
			\draw  [fill={rgb, 255:red, 0; green, 0; blue, 0 }  ,fill opacity=1 ] (221.59,188.09) .. controls (221.59,185.24) and (223.79,182.93) .. (226.51,182.93) .. controls (229.23,182.93) and (231.43,185.24) .. (231.43,188.09) .. controls (231.43,190.94) and (229.23,193.25) .. (226.51,193.25) .. controls (223.79,193.25) and (221.59,190.94) .. (221.59,188.09) -- cycle ;
			%Straight Lines [id:da5536433068888074] 
			\draw    (200.51,61.09) -- (399.51,60.09) ;
			%Straight Lines [id:da6605071399166285] 
			\draw    (250.51,127.09) -- (351.51,126.09) ;
			%Straight Lines [id:da04585793937113358] 
			\draw    (226.51,188.09) -- (377.51,187.09) ;
			%Straight Lines [id:da24061830532102935] 
			\draw    (200.51,61.09) -- (250.51,127.09) ;
			%Straight Lines [id:da7771984778678867] 
			\draw    (399.51,60.09) -- (351.51,126.09) ;
			%Straight Lines [id:da5941149878869683] 
			\draw    (250.51,127.09) -- (226.51,188.09) ;
			%Straight Lines [id:da5769562824790726] 
			\draw    (351.51,126.09) -- (377.51,187.09) ;
			
			% Text Node
			\draw (396,30.4) node [anchor=north west][inner sep=0.75pt]    {$v_5$};
			% Text Node
			\draw (346,30.4) node [anchor=north west][inner sep=0.75pt]    {$v_4$};
			% Text Node
			\draw (295,30.4) node [anchor=north west][inner sep=0.75pt]    {$v_3$};
			% Text Node
			\draw (247,30.4) node [anchor=north west][inner sep=0.75pt]    {$v_2$};
			% Text Node
			\draw (194,30.4) node [anchor=north west][inner sep=0.75pt]    {$v_1$};
			% Text Node
			\draw (219,117.4) node [anchor=north west][inner sep=0.75pt]    {$v_6$};
			% Text Node
			\draw (296,93.4) node [anchor=north west][inner sep=0.75pt]    {$v_7$};
			% Text Node
			\draw (374,115.4) node [anchor=north west][inner sep=0.75pt]    {$v_8$};
			% Text Node
			\draw (370,202.4) node [anchor=north west][inner sep=0.75pt]    {$v_{12}$};
			% Text Node
			\draw (319,202.4) node [anchor=north west][inner sep=0.75pt]    {$v_{11}$};
			% Text Node
			\draw (271,204.4) node [anchor=north west][inner sep=0.75pt]    {$v_{10}$};
			% Text Node
			\draw (218,202.4) node [anchor=north west][inner sep=0.75pt]    {$v_9$};

		\end{tikzpicture}
	}
	\caption{Bicyclic graph representing the logistical problem network}
	\label{application}
\end{figure}
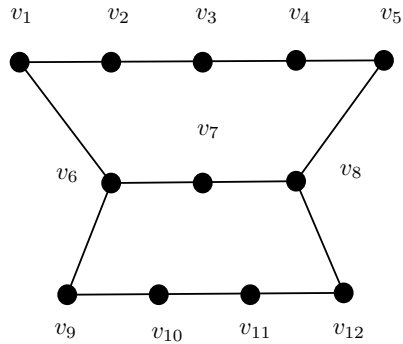

By representing the transportation network as a graph, we can develop solutions that apply broadly, transcending individual cases to encompass a
variety of logistical scenarios.

Since the number of vertices in first path, $v_1 v_2 v_3 v_4 v_5$, are more than both other paths, we can apply Theorem 6 from above. Moreover, since $q-r=2-3=-1<2$, where $q$ is the the number of vertices second path and $r$ is the number of vertices of first path, Part 1 of Theorem 6 is applicable and we get that $\{v_1, v_4\}$ are the metric basis in this case. This tells us that the fields $1$ and $4$ are the landmarks of the logistics network and the relative position of the whole network can be easily identified using these landmarks.

Let us now consider that both transporters are equipped with advanced onboard computers and GPS systems. These systems must communicate in real-time to synchronize operations, ensuring that both transporters start at the same time and avoid overlapping coverage. Weather changes can affect the suitability of either transporter. For instance, strong winds might limit aerial drone operations, while muddy conditions might restrict ground trucks. Transport efficiency must be optimized to reduce costs and environmental impact. Both transporters need to track the amount of crops transported and adjust their operations based on real-time feedback from the fields.

This is an example of a real world logistics problem where computers need to make autonomous decisions based on real world inputs. Since, there is a need to give a unique identifier to every field, as well as to compute the relative distance of the fields, and the positions of the transporters at any given time, the metric basis (landmarks) is an effective way to calculate all these parameters. The distance vectors in this case not only give a unique identifier to every field, they also provide the relative distances with respect to the chosen landmarks (metric basis). Feeding this data into a machine learning algorithm, a model can be easily developed which can efficiently solve this logistics problem.

This problem not only illustrates the logistical challenges of crop transportation in a diverse agricultural network but also highlights the role of advanced technology and strategic business management in overcoming these challenges.

\textbf{Further Expansion of the Proposed Solutions:}

$1.$ As the farm expands, the complexity of managing the transportation logistics increases. In such a scenario, the machines can use metric dimensions and metric basis to effectively use landmarks to identify all fields in the network and communicate their data to the other partner. This concept is further explained in \cite{khuller1996landmarks}.

$2.$ By linking the problem of optimizing agricultural supply chain logistics with a decision-making process involving alternatives (Super Truck and aerial Agri Drone) and attributes (accessibility, efficiency, cost, environmental impact), we can effectively address the logistical challenges and complexities of transporting crops in a diverse agricultural network using complex fuzzy intelligent decision modeling \cite{wang2024complex}.

\section{SUMMARY}

We worked with the resolving set and metric dimension of Bicyclic graphs of type III and proved that $\beta({\C_{p,q,r}})$ is a constant with values $2$ or $3$ depending on the values of $p,q,r$.

Metric dimension problem for these types of graphs was already discussed in \cite{10142411} using the term $\Theta$-graph. However, due to some limitations, the authors were unable to consider all the cases, which produced some minor mistakes in the proposed results. This article is a parallel research into the same problem and aims to rectify the mistakes present in that research.

We also proposed an application of these metric basis to optimize a supply chain logistics problem for a complex network, with autonomous entities operating inside the network.

\section{CONCLUSION}

It was proved in \cite{math11040869} that the metric dimension of bicyclic graphs of type-I and II is constant. This article concludes that the metric dimension of bicyclic graphs of type-III is also constant. Specifically, when $\P_1$ and $\P_2$ are of same length and $\P_3$ is $2$ less than both $\P_1, \P_2$, the metric dimension will be $3$. Similarly, when $\P_1$ and $\P_3$ are of equal length and $\P_2$ is $2$ or $4$ vertices longer than both $\P_1,\P_3$, the metric dimension is $3$. Except these $3$ cases, for all other possibilities of $\P_1$, $\P_2$ and $\P_3$, the metric dimension always remains $2$.

%It should be noted that the relationship of $p,q$ and $r$, namely, $p \geq q$ with $r$ unrestricted, actually solves all possible cases of bicyclic graphs of type-III by isomprphism. As an example, let us consider a case which is not covered here, say, $p <q$ and $p<r$. Since the relationship of $q$ and $r$ is not specified, by trichotomy property, $q<r$, $q=r$ and $q>r$ are all valid possibilities. It is already mentioned that $C_{p,q,r} \simeq C_{r,q,p}$ for fixed $p,r$. These facts enable us to state the following statements.

%\begin{itemize}
%	\item When $p<q$, $p<r$ and $q<r$, we get $\max\{p,q\}<r$. This case was solved in Theorem \ref{Theorem 3}.
%	\item When $p<q$, $p<r$ and $q=r$. Using the isomorphism, we see that this was solved in Theorem \ref{Theorem 2}.
%	\item When $p<q$, $p<r$ and $q>r$ ($r<q$), we get $p<r<q$ or $\max\{p,r\}<q$ which was solved in Theorem \ref{Theorem 1}.
%\end{itemize}

We also mention that unicyclic graphs have enjoyed a huge interest in terms of metric dimension problems compared to bicyclic graphs. Following questions/ problems are a future avenue to continue the current research.

Problem 1: Studying other variants of metric dimension for these graphs, e.g., edge metric dimensions, local metric dimensions, fault tolerant metric dimensions etc.

Problem 2: Extending bicyclic graphs to tricyclic graphs and calculating their metric dimensions.

Problem 3: Applying graph operations on $\C_{p,q,r}$, e.g., Cartesian product and join operation, with other graph families and calculating their metric dimensions.

Problem 4: Studying the changes in $\beta(\C_{p,q,r})$ by removing vertex/edge.

%Bicyclic graphs appear in many fields of sciences. e. g., as bicyclic compounds \cite{sorrell2006organic} in chemistry and as complex networks \cite{ma2016wiener} in physical, biological, and social sciences. Moreover, graph distance profiles are used for machine learning purposes in graph neural networks \cite{razaghi2020supervised}. Combining these together, we can easily generate models of chemicals and networks, using bicyclic graphs, and apply distance based machine learning algorithms to understand their properties.

%%%%%%%%%%%%%%%%%%%%%%%%%%%%%%%%%%%%%%%%%%ma2016wiener
\vspace{6pt} 

%%%%%%%%%%%%%%%%%%%%%%%%%%%%%%%%%%%%%%%%%%

\textbf{Author contribution statement} 

All authors listed have significantly contributed to the development and the writing of this article. 

\textbf{Data availability statement} 

No data was used for the research described in the article. 

\textbf{Additional information} 

No additional information is available for this paper. 

\textbf{Declaration of competing interest }

The authors declare that they have no known competing financial interests or personal relationships that could have appeared to 
influence the work reported in this paper.

\textbf{Acknowledgment}
This work was supported by the National Social Science Foundation of China (no. 23BJY010) and the Shandong Province Social Science Planning Subjects (no. 23CJJJ31).

%%%%%%%%%%%%%%%%%%%%%%%%%%%%%%%%%%%%%%%%%%

\bibliographystyle{ieeetran}

\bibliography{paper}{}

\end{document}